\documentclass[11pt,a4paper,twoside]{amsart}

\usepackage[T1]{fontenc}
\usepackage[utf8]{inputenc}
\usepackage[english]{babel}
\usepackage{float}
\usepackage{xspace}
\usepackage{lmodern}
\DeclareFontFamily{U}{wncyr}{}
\DeclareFontShape{U}{wncyr}{m}{n}{<->wncyr10}{}
\DeclareFontShape{U}{wncyr}{m}{it}{<->wncyi10}{}
\DeclareFontShape{U}{wncyr}{m}{sc}{<->wncysc10}{}
\DeclareFontShape{U}{wncyr}{b}{n}{<->wncyb10}{}
\DeclareTextCommand{\guillemotleft}{T1}{%
  {\fontencoding{U}\fontfamily{wncyr}\selectfont\symbol{"3C}}%
}
\DeclareTextCommand{\guillemotright}{T1}{%
  {\fontencoding{U}\fontfamily{wncyr}\selectfont\symbol{"3E}}%
}

\usepackage{amsmath,amsfonts,amsthm,amssymb}
\usepackage{tikz-cd}

\usepackage{graphicx}
\usepackage{tikz}
\usepackage{ifthen}
\usetikzlibrary{arrows,calc,intersections,patterns}
\usepackage{rotating} 

\usepackage{caption}
\usepackage{subcaption}

\usepackage{enumitem}

\allowdisplaybreaks[3] 
\usepackage{tikz-cd}
\usepackage{mathtools}
\usepackage[]{hyperref}
\hypersetup{
    colorlinks=true,       
    linkcolor=blue,          
    citecolor=magenta,        
    filecolor=magenta,      
    urlcolor=cyan           
}
\usepackage{xcolor}

\usepackage[top=3.5cm,bottom=3cm,left=2.5cm,right=2.5cm]{geometry}

\newtheorem{theorem}{Theorem}[section]
\newtheorem{proposition}[theorem]{Proposition}
\newtheorem{corollary}[theorem]{Corollary}
\newtheorem{lemma}[theorem]{Lemma}

\theoremstyle{definition}
\newtheorem{definition}[theorem]{Definition}

\theoremstyle{remark}
\newtheorem{remark}[theorem]{Remark}

\numberwithin{equation}{section}

\DeclareRobustCommand{\SkipTocEntry}[5]{}

\usepackage{siunitx}
\def\longrightharpoonup{-\negthickspace \! \rightharpoonup}
\begin{document}

\title[Almost Sure Existence of Global Solutions for Wave equations]{Almost Sure Existence of Global Solutions for Supercritical Semilinear Wave Equations}

\author{Mickaël Latocca}
\address{Département de Mathématiques et Applications, Ecole Normale Supérieure, 45 rue d'Ulm 75005 Paris, France}
\email{mickael.latocca@ens.fr}

\date{\today}

\maketitle

\begin{abstract}
\begin{sloppypar}We prove that for almost every initial data $(u_0,u_1) \in H^s \times H^{s-1}$ with $s > \frac{p-3}{p-1}$ there exists a global weak solution to the supercritical semilinear wave equation ${\partial _t^2u - \Delta u +|u|^{p-1}u=0}$ where $p>5$, in both $\mathbb{R}^3$ and $\mathbb{T}^3$. This improves in a probabilistic framework the classical result of Strauss~\cite{strauss} who proved global existence of weak solutions associated to $H^1 \times L^2$ initial data. The proof relies on techniques introduced by Oh and Pocovnicu in \cite{ohPocovnicu} based on the pioneer work of Burq and Tzvetkov in \cite{burqTzvetkov2}. We also improve the global well-posedness result in \cite{sunXia} for the subcritical regime $p<5$ to the endpoint $s=\frac{p-3}{p-1}$. 
\end{sloppypar}
\end{abstract}
\tableofcontents 
\section{Introduction}

\subsection{Supercritical semilinear wave equations} We consider the Cauchy problem for the energy-supercritical defocusing semilinear equation in dimension $3$, that is for $p > 5$ and $s \geqslant 0$:
\begin{equation}
\tag{SLW$_p$}
\label{1SLW}
    \left \{
    \begin{array}{cc}
         \partial_t^2 u - \Delta u +|u|^{p-1}u=0 \\
         (u(0), \partial _t u (0))=(u_0,u_1) \in H^s(U) \times H^{s-1}(U) \,, 
    \end{array}
    \right.
\end{equation}
where $u(t)$ is a real-valued function defined on $U = \mathbb{R}^{3}$ or $\mathbb{T}^3$. We also consider the associate linear wave equation: 
\begin{equation}
\tag{LW}
\label{1LW}
    \left \{
    \begin{array}{cc}
         \partial_t^2 z - \Delta z = 0 \\
         (z(0), \partial _t z (0))=(z_0,z_1) \in H^s(U) \times H^{s-1}(U)\,.  
    \end{array}
    \right.
\end{equation}
The formal conserved energy for a solution $u$ to \eqref{1SLW} is \[E(u(t)) := \int_{U} \left( \frac{|\partial _t u (t,x)|^2}{2} + \frac{|\nabla u (t,x)|^2}{2} + \frac{|u(t,x)|^{p+1}}{p+1}\right) \, \mathrm{d}x \,.\]
Moreover~\eqref{1SLW} is known to be invariant under the dilation symmetry \[u(t,x) \mapsto u_{\lambda} (t,x) = \lambda ^{\frac{2}{p-1}}u(\lambda t,\lambda x) \,.\] A necessary condition for a function $u$ to belong to the \textit{energy space}, \textit{\textit{i.e.}} $E(u(t))< \infty$ is that $(u(0),\partial _t u (0)) \in \dot{H}^1(U) \times L^2(U)$. We observe that 
\begin{align*}
     \|(u_{\lambda}(0),\partial _t u_{\lambda} (0))\|_{\dot{H}^1(U) \times L^2(U)} &= \lambda^{\frac{2}{p-1}-\frac{1}{2}} \|(u(0),\partial _t u (0))\|_{\dot{H}^1(U) \times L^2(U)}\\
    & \underset{\lambda \to 0}{\gg} \|(u(0),\partial _t u (0))\|_{\dot{H}^1(U) \times L^2(U)}\,,
\end{align*}
since $p>5$, which explains why for such $p$, \eqref{1SLW} is called \textit{energy-supercritical}. 

We first recall the classical result of existence of \textit{weak solutions} to~\eqref{1SLW}. 

\begin{theorem}[Strauss, 1970, \cite{strauss}]\label{1straussTheo} Let $f$ be a real smooth function and $F$ be an antiderivative of $f$. Assume that $F(v) \gtrsim - |v|^2$ and  \[\frac{|F(u)|}{|f(u)|} \to \infty \text{ as } |u| \to \infty \,.\] Let $(u_0,u_1) \in H^1 (\mathbb{R}^3) \times L^2(\mathbb{R}^3)$ satisfying \[E(u_0,u_1) := \int_{\mathbb{R}^3} \left( \frac{|u_1|^2}{2} + \frac{|\nabla u_0 |^2}{2} + F(u_0)\right) \, \mathrm{d}x < \infty\,.\]
Then the equation 
\begin{equation}
    \label{1generalWaveEquation}
    \left\{ \begin{array}{ccc}
    \partial_t^2 u - \Delta u + f(u) &=&0  \\
     (u(0),\partial _t u(0)) &=&(u_0,u_1) \in H^1 \times L^2 \,,
\end{array}\right.
\end{equation}
admits a weak solution, that is a distributional solution $u : \mathbb{R} \to H^1(\mathbb{R}^3) \times L^2(\mathbb{R}^3)$ which is weakly continuous in time and such that \[ E(u(t),\partial _t u(t)) \leqslant E(u_0,u_1) \text{ for every } t \in \mathbb{R}\,.\] 
\end{theorem}

In the following we will seek for global solutions to \eqref{1SLW} in the following sense. 

\begin{definition}[Weak solutions for \eqref{1SLW}] \label{1weakSolution} A function $u : \mathbb{R} \times U \to \mathbb{R}$ is said to be a \textit{weak solution} to \eqref{1SLW} with initial data $(u_0,u_1)\in H^s(U) \times H^{s-1}(U)$ if and only if one can write $u=z+v$ where $z$ is a strong solution to \eqref{1LW} with initial data $(u_0,u_1)$ and $v$ is such that $(v(0),\partial _t v(0))=(0,0)$ and satisfies for every $T >0$, $v\in H^1((-T,T) \times U) \cap L^{p+1}((-T,T) \times U)$, and for every compactly supported $\varphi \in \mathcal{C}^2(\mathbb{R} \times U)$:
\[ \int_{\mathbb{R}} \int_{U} \left( \partial _t v(t) \partial _t \varphi (t) - \nabla v(t) \cdot \nabla \varphi (t) - (z(t)+v(t)) |z(t)+v(t)|^{p-1} \varphi(t)\right)  \, \mathrm{d}x \, \mathrm{d}t = 0\,.\]
\end{definition}

\begin{remark} Note that this definition differs from the definition of weak solutions in Theorem~\ref{1straussTheo}. In our setting we ask for the solution $u$ to be written in the form $u=z+v$, where $z$ solves the associate linear problem~\eqref{1LW}: the reason why we ask for such a decomposition will appear clearly in the proof of the main result, Theorem~\ref{1mainTheorem}. Note that these weak solutions are \textit{a fortiori} weak solution as in Theorem~\ref{1straussTheo}.  
\end{remark}

In contrast, we recall the notion of strong solution that we will use, slightly different than usual definitions, adapted to a decomposition $u=z+v$. More precisely, we have the following. 

\begin{definition}[Strong solutions for~\eqref{1SLW}]\label{1defStrong} Let $T>0$. A function $u : (-T,T) \times U \to \mathbb{R}$ is said to be a \textit{strong solution} to~\eqref{1SLW} on the time interval $(-T,T)$, with initial data $(u_0,u_1) \in H^s(U) \times H^{s-1}(U)$ if there is a decomposition $u=z+v$ where $(z,\partial _tz) \in \mathcal{C}^0((-T,T),H^s(U) \times H^{s-1}(U))$ is a solution to~\eqref{1LW} associated to $(u_0,u_1)$ and $(v,\partial_t v)$ belongs to $\mathcal{C}^0(\mathbb{R},H^s(U) \cap L^{p+1}(U)) \times \mathcal{C}^0(\mathbb{R},H^{s-1}(U))$ and is a solution to 
\[\partial_t^2v-\Delta v + |z+v|^{p-1}(z+v)=0\,,\]
with initial condition $(v_0,v_1)=(0,0)$. Such a solution is \textit{global} if one can take $T=\infty$. 
\end{definition}

\begin{remark} We remark that these \textit{strong} solutions are such that \[(u,\partial_tu) \in \mathcal{C}^0((-T,T),H^s) \times \mathcal{C}^0((-T,T),H^{s-1})\,,\]
which can be considered to be the usual space for \textit{strong} solutions. We will see that, however, differences appear when one considers the uniqueness of such solutions, see Theorem~\ref{1th2}. 
\end{remark}

\subsection{Previous works on probabilistic well-posedness for wave equations} Our purpose is to construct solutions to~\eqref{1SLW}, in the sense of Definition~\ref{1weakSolution} using a probabilistic method when initial data are below the energy space. Indeed, up to the knowledge of the author, no global existence result is known for~\eqref{1SLW}, $p>5$ and initial data $(u_0,u_1) \in H^s \times H^{s-1}$ with~$s<1$.

Probabilistic methods have been implemented in order to construct solutions to~\eqref{1SLW} associated to initial data below the energy space. As a result, the local and global well-posedness theory have been widely improved. We briefly recall the existing results in the context of semilinear wave equations, our list being not exhaustive.  

The probabilistic well-posedness theory goes back to J. Bourgain who proved global existence for the two-dimensional nonlinear Schrödinger equation in~\cite{bourgain}. Building on Bourgain's ideas, N. Burq and N. Tzvetkov published a series of two articles~\cite{burqTzvetkov2,burqTzvetkov3},  introducing a randomization procedure that allows to choose random initial data in Lebesgue and Sobolev spaces. They developed the local and global probabilistic well-posedness theory. They later considered the global well-posedness of~\eqref{1SLW} for $p=3$ in~\cite{burqTzvetkov} proving that, although for $s<\frac12$, \eqref{1SLW} is ill-posed, there exist unique global solutions for almost every initial data in $H^s(\mathbb{T}^3) \times H^{s-1}(\mathbb{T}^3)$ as soon as $s \geqslant 0$. The proof relies on a probabilistic improvement of the Strichartz estimates. In~\cite{burq}, Burq-Thomann-Tzvetkov considered~\eqref{1SLW} in higher dimensions and proved the almost sure existence of global infinite energy solutions of~\eqref{1SLW}, for $p=3$ in $\mathbb{T}^d$, $d \geqslant 3$. Their argument use compactness techniques just like the ones presented in this article. We mention that in the context of the Navier-Stokes equation, A.R. Nahmod, N. Pavlovi\'c  and G. Staffilani proved existence of global weak solutions almost surely in~\cite{nahmodPavlovicStaffilani}.

The work of Lührman-Mendelson in~\cite{lurmanMendelson,lurmannMenselsonBis} deals with the global well-posedness theory for~\eqref{1SLW} in the case $3<p<5$. They prove an almost-sure global well-posedness result associated to initial data \[(u_0,u_1) \in H^s(\mathbb{R}^3) \times H^{s-1}(\mathbb{R}^3) \text{ as long as } s > \frac{p^3+5p^2-11p-3}{9p^2-6p-3}\] which improves the deterministic theory when $\frac{1}{4}(7+\sqrt{73}) \simeq 3.88 < p < 5$. In~\cite{lurmannMenselsonBis} they improved their result to $\frac{p-1}{p+1}<s<1$ using Oh-Pocovnicu's ideas from \cite{ohPocovnicu}. 

In~\cite{pocovnicu} O. Pocovnicu proved almost-sure global well-posedness for the energy critical wave equation \eqref{1SLW}, that is $p=5$, in the euclidean space $\mathbb{R}^d$ of dimension $d=4,5$. The proof relies on the deterministic perturbation theory for critical dispersive equations as well as the probabilistic improvements of the Strichartz estimates coming from the work of Burq-Tzvetkov. With some more efforts in the domains $\mathbb{R}^3$ and $\mathbb{T}^3$ the global well-posedness theory for~\eqref{1SLW}, $p=5$ has been treated in the joint work of Oh-Pocovnicu in~\cite{ohPocovnicu,ohPocovnicu2}. In their proof they used a new energy estimate and a new probabilistic Strichartz estimate. Their result shows that almost-sure global well-posedness in known to hold for initial data in $H^s \times H^{s-1}$, $s>\frac{1}{2}$. 

The global well-posedness theory for~\eqref{1SLW} and when $3<p<5$ was then studied in the work of Sun-Xia, in~\cite{sunXia}. They proved global existence and uniqueness for $s > \frac{p-3}{p-1}$ interpolating between the results of Oh-Pocovnicu \cite{ohPocovnicu} and Burq-Tzvetkov \cite{burqTzvetkov}.

Finally it should be mentioned that in the context of the semilinear Shrödinger equation, similar probabilistic well-posedness have been investigated, we refer to \cite{benyiOhPocovnicu2,okamotoOhPocovnicu} for the euclidean setting and \cite{bourgainBulut1,bourgainBulut2} for the setting of the unit ball. 

\subsection{Main results and notations}

\subsubsection{Statement of the main results} Let $(\Omega, \mathcal{F},\mathbb{P})$ be a probability space, and a randomization map $(u_0,u_1)\longmapsto (u_0^{\omega},u_1^{\omega})$ that will be described in Section~\ref{1sousSectionProba}, see \eqref{1rand1} and \eqref{1rand2}. Let $U$ be an open set. The measure $\mu_{(u_0,u_1)}$ is defined as the pushforward probability measure of $\mathbb{P}$ by the above randomization map. We define

\[\mathcal{M}^s(U):=\bigcup_{(u_0,u_1)\in H^s(U) \times H^{s-1}(U)} \left\{\mu_{(u_0,u_1)}\right\}\,.\] 

We now state our results. The first is an existence result in the supercritical case. 

\begin{theorem}\label{1mainTheorem} Let $p > 5$, $s > \frac{p-3}{p-1}$ and $U$ which stands for $\mathbb{R}^3$ or $\mathbb{T}^3$. Let $\mu \in \mathcal{M}^s(U)$. Then for $\mu$ almost every $(u_0,u_1) \in H^s(U) \times H^{s-1}(U)$ there exists a global weak solution $u$ to \eqref{1SLW}, in the sense of Definition~\ref{1weakSolution}. 
\end{theorem}

\begin{remark} Note that no information is given concerning the uniqueness of the solutions constructed. 
\end{remark}

The solutions constructed in Theorem~\ref{1mainTheorem} enjoy additional properties:

\begin{corollary}\label{1coroAdditional} Under the hypothesis of Theorem~\ref{1mainTheorem}, and given $\mu \in \mathcal{M}^s(U)$, there exists a set ${\Sigma \subset H^s(U) \times H^{s-1}(U)}$ of full $\mu$-measure which is invariant under the flow of~\eqref{1SLW}. If~$s<1$ then for every initial data $(u_0,u_1) \in \Sigma$, the solutions $u$ constructed by Theorem\ref{1mainTheorem} satisfy $(u,\partial _t u) \in \mathcal{C}^0(\mathbb{R},H^s(U)\times H^{s-1}(U))$
\end{corollary}

\begin{corollary}\label{1finiteSpeed} For $U=\mathbb{R}^3$, the solutions constructed by Theorem~\ref{1mainTheorem} enjoy the finite speed of propagation property with speed at most $1$.
\end{corollary}

The next result is an extension of Theorem~1.2 from~\cite{sunXia} to the endpoint $s=\frac{p-3}{p-1}$. 

\begin{theorem}\label{1th2} Let $p < 5$, $s:=\frac{p-3}{p-1}$ and $U$ which stands for $\mathbb{R}^3$ or $\mathbb{T}^3$. Let $\mu \in \mathcal{M}^s$. Then for almost every $(u_0,u_1) \in H^s(U) \times H^{s-1}(U)$ there exists a unique global \emph{strong} solution to~\eqref{1SLW} in the sense of Definition~\ref{1defStrong}, where uniqueness for $(u, \partial _tu)$ holds in the space
\[S(\cdot)(u_0,u_1)+ \mathcal{C}^0(\mathbb{R},H^s(U) \cap L^{p+1}(U)) \times \mathcal{C}^0(\mathbb{R},H^{s-1}(U))\,,\]
with $S(\cdot)(u_0,u_1)$ standing for the solution to~\eqref{1LW} associated to $(u_0,u_1)$. 
\end{theorem}

Again a consequence of the proof of Theorem~\ref{1th2} is the following: 

\begin{corollary}\label{1coroth2} Under the hypothesis of Theorem~\ref{1th2}, and given $\mu \in \mathcal{M}^s$, there exists a set $\Sigma \subset H^s(U) \times H^{s-1}(U)$ of full $\mu$-measure which is invariant under the flow of~\eqref{1SLW}.
\end{corollary}

\begin{remark} One can also prove probabilistic continuous dependence of the flow in the sense of \cite{burqTzvetkov}, using the proof given in \cite{pocovnicu}, Theorem~1.4 and Remark~1.6~(iii) but this paper does not focus on that matter. 
\end{remark}

\begin{remark} Combining these results with the existing results in the case $p \leqslant 5$, see \cite{burqTzvetkov,ohPocovnicu,ohPocovnicu2,sunXia} yields the following classification: 
\begin{enumerate}[label=(\textit{\roman*})]
    \item In the \textit{energy sub-critical} setting \textit{i.e} $p \in [3,5)$, there exists a unique strong solution to \eqref{1SLW} for amlost every initial data in $H^s(U)\times H^{s-1}(U)$ as soon as $s \geqslant \frac{p-3}{p-1}$. The case $p=3$ is treated entirely in \cite{burqTzvetkov} in the case of the torus, but similar arguments work in the euclidean setting. The case $p\in(3,5)$ and $s>\frac{p-3}{p-1}$ is treated in \cite{sunXia}, and the case $p\in(3,5)$, $s=\frac{p-3}{p-1}$ is Theorem~\ref{1th2}. In \cite{burqTzvetkov}, additional results of continuous dependence and flow-invariant set are proven in the case $p=3, s \geqslant 0$ which remain valid in the case $p \in (3,5)$, $s\geqslant \frac{p-3}{p-1}$. 
    \item In the \textit{energy critical} setting $p=5$, there exists a unique strong solution to \eqref{1SLW} with initial data in $H^s(U)\times H^{s-1}(U)$ as soon as $s > \frac{1}{2}$. This result is proven in \cite{ohPocovnicu,ohPocovnicu2} both in $\mathbb{R}^3$ and $\mathbb{T}^3$ and comes with a continuous dependence result and flow-invariant set construction.
    \item In the \textit{energy super-critical} setting $p>5$ there exists a global strong solution as long as $s \in (\frac{p-3}{p-1},1)$ : this is Theorem~\ref{1mainTheorem}. For $s>1$ there still exists a weak solution, in the sense of Definition~\ref{1weakSolution}. In both cases such solutions can be constructed in a flow-invariant way, this is Corollary~\ref{1coroAdditional}.
\end{enumerate}
\end{remark}

\begin{remark} One can wonder if it is possible, for $p > 5$ (resp. $p=5$), to extend the result of Theorem~\ref{1mainTheorem} (resp. Theorem~\ref{1th2}) to the endpoint $s=\frac{p-3}{p-1}$. As it will be explained in Remark~\ref{1remImpossible}, this endpoint remains out of reach of the techniques presented on this article. Thus the extension to the endpoint proven in Theorem~\ref{1th2} may be viewed as subcritical. 
\end{remark}

\subsubsection{General notation} $(\Omega, \mathcal{F}, \mathbb{P})$ is called a probability space if $\Omega$ is a set and $\mathcal{F}$ is a $\sigma$-algebra, endowed with a probability measure $\mathbb{P}$. The expectation of a random variable $X$ will be denoted by $\mathbb{E}[X]$. 

The notation $A \lesssim B$ means that there exists a constant $C$ such that $A \leqslant CB$. The notation~$A \lesssim_{x} B$ is used to specify that the constant $C$ depends on $x$. 

For a real number $x$, $\lfloor x \rfloor$ (resp. $\lceil x \rceil$) denotes the lower integer part (resp. upper integral part).

We adopt widely used notations for functional spaces: $\mathcal{C}^k$ denotes the set of $k$ differentiable functions with continuous derivatives up to order $k$, $L^p$ stands for the Lebesgue spaces and $W^{s,p}:=\{u, (\operatorname{Id}-\Delta)^{s/2} u \in L^p\}$ (resp. $\dot W^{s,p}$) denotes the usual nonhomogeneous Sobolev spaces (resp. homogeneous spaces). The hilbertian Sobolev spaces are denoted $H^s:=W^{s,2}$. Besov spaces are denoted $B^{s}_{p,r}$, see Appendix~\ref{1appendixA} for more details. The Fourier transform of a function $f$ is denoted by $\hat f$ or equivalently $\mathcal{F}(f)$. $\mathcal{D}$ denotes the space of test functions, that is $\mathcal{C}^{\infty}$ functions with compact support. $\mathcal{S}'$ stands for the space of tempered distributions. If $X$ is a Banach space we often write $L^pX$ or $L^p_TX$ for $L^p((0,T),X)$. For $x \in \mathbb{R}^d$, $\langle x \rangle = (1+|x|^2)^{1/2}$, and $\langle \nabla \rangle$ denotes the Fourier multiplier of symbol $\langle \xi \rangle$.

$S^m:=S^m_{1,0}$ denotes the class of classical symbols of order $m \in \mathbb{R}$, that is smooth functions $a : \mathbb{R}^{d} \times \mathbb{R}^{d} \to \mathbb{C}$ such that \[| \partial^{\alpha} _{\xi}\partial _x^{\beta} a(x,\xi)| \lesssim_{\alpha, \beta} \langle \xi \rangle^{m-|\alpha|}, \text{ for every } \alpha \in \mathbb{N}^{d} \text{ and } x,\xi \in \mathbb{R}^d\,.\]

In the rest of this article $\mathcal{H}^s(U)$ will be used as a shorthand notation for $H^s(U) \times H^{s-1}(U)$.

\subsection{Outline of the paper, heuristic arguments} The purpose of this section is to provide the reader with heuristic arguments that will help to understand the main ideas behind the proofs of the two main results. Note that this section is not mathematically needed in order to follow the rigorous proofs of the result.

\subsubsection{Energy estimates} In this paragraph we set $U= \mathbb{R}^3$ and drop the reference to it. As in~\cite{burqTzvetkov,ohPocovnicu}, the method to construct global solutions to \eqref{1SLW} is to seek for solutions~$u$ of the form $u(t)=z(t)+v(t)$ with initial data $v(0)=\partial _t v(0)=0$, and $z$ being the solution to~\eqref{1LW} associated to initial data $(u_0,u_1) \in \mathcal{H}^s$. Note that $z$ is globally defined. It is thus sufficient to prove that $v$ globally solves an equation in the energy space $\mathcal{H}^1$. A direct computation shows that $v$ formally satisfies \[\partial _t^2 v - \Delta v + (z+v)^p=0 \,.\] 

We now assume that $v$ locally solves this equation in $\mathcal{H}^1$ and that the local well-posedness result comes with a blow-up criteria that only depends on the size of $\|(v(t),\partial _t v(t))\|_{\mathcal{H}^1}$. In this case, proving global existence $v$ reduces to proving that the (not conserved) nonlinear energy \[E(t):=\frac{\|\partial_t v(t)\|_{L^2}^2}{2}+\frac{\|\nabla v(t)\|_{L^2}^2}{2}+\frac{\|v(t)\|^{p+1}_{L^{p+1}}}{p+1}\] is bounded on every time interval. 

In order to do so, the standard way is to estimate $E(t)$ using a Grönwall-type estimate and hope for a sublinear estimate that will give non-blowup for $E(t)$. We first begin by writing that \[E'(t)= \int_{\mathbb{R}^3} \partial _t v(t) \left(|z(t)+v(t)|^{p-1}(z(t)+v(t))-|v(t)|^{p-1}v(t)\right) \, \mathrm{d}x \,.\] 
In the following we will provide a rough argument, ignoring lower order terms and using fractional integration by parts. The terms in powers of $z$ will be estimated using the probabilistic Strichartz estimates from Proposition~\ref{1propStrichartz} and Proposition~\ref{1propStrichartzBesov} so they constitute the ``good part'' when developing the quantity $(z+v)^p-v^p$, if we assume $p$ to be an odd integer for instance. These considerations lead to the ``worse order approximation'' $(z+v)^p-v^p \simeq v^{p-1}z + \text{(better terms)}$. The other terms are expected to be handled in an easier way so that we write : \[E'(t) \simeq \int_{\mathbb{R}^3} \partial _t v\, v^{p-1}z \, \mathrm{d}x + \underbrace{\text{(better terms)}}_{\lesssim E(t)} \,.\] 
From now we will just dismiss the better terms and study the worse term. 

\begin{remark} A crude estimate using the Hölder inequality, putting $z \in L^{\infty}$, $\partial _t v \in L^2$ and $v^{p-1} \in L^{\frac{p+1}{p-1}}$ leads to \[E'(t) \lesssim E(t')^{\frac{1}{2}+\frac{p-1}{p+1}}\] which is sublinear if $\frac{1}{2} + \frac{p-1}{p+1} \leqslant 1$ \textit{i.e.} $p \leqslant 3$. Thus such an argument would provide an energy estimate that does not blow-up for $p \leqslant 3$ and $s>0$. This was the idea behind the energy estimates in \cite{burqTzvetkov,pocovnicu}.
\end{remark}

In order to obtain energy estimates for $p=5$, Oh-Pocovnicu introduced in \cite{ohPocovnicu} an appropriate method that we describe: if one accepts a loss of regularity for the initial data then one can transfer time regularity into space regularity using an integration by part in time, and properties of the wave equation, which state that roughly $\partial _t z \simeq \nabla z$. 
Since $E(0)=0$ we have after integration in time: \[E(t) = \int_0^t E'(t')\, \mathrm{d}t' \simeq \int_0^t \int_{\mathbb{R}^3} \partial _t v v^{p-1}z \, \mathrm{d}x\, \mathrm{d}t'\,.\]
Then using that $\partial _t v v^{p-1} \sim \partial _t (v^p)$ an integration by parts yields 
\begin{align*}
    E(t) & \simeq \int_{\mathbb{R}^3} \int_0^t v(t')^p \partial _t z(t') \, \mathrm{d}t' \, \mathrm{d}x \\
    & \simeq \int_{\mathbb{R}^3} \int_0^t v(t')^p \nabla z(t') \, \mathrm{d}t' \, \mathrm{d}x \,.
\end{align*}
Pick $s\in[0,1]$ which will be chosen later. We write that $\nabla z = \nabla^{1-s}\nabla^s z$ and integrate by parts in space with the operator $\nabla ^{1-s}$, neglecting the boundary terms:  \[ E(t) \simeq \int_0^t \int_{\mathbb{R}^3} \nabla ^{1-s}(v(t')^p) \nabla^s z(t') \, \mathrm{d}x \, \mathrm{d}t' \,.\] 
We expect that $\nabla ^{1-s}(v^p) \simeq v^{p-1} \nabla ^{1-s}v$ so that Hölder's inequality yields 
\begin{align}
    E(t) & \label{1eqEnergy}\simeq  \int_0^t \int_{\mathbb{R}^3} \nabla ^{1-s}v(t') v(t')^{p-1} \nabla^s z(t') \, \mathrm{d}x \, \mathrm{d}t' \\
    & \lesssim \|\nabla ^s z\|_{L^{\infty}_{T,x}}\int_0^t E(t')^{\frac{p-1}{p+1}} \|v(t')\|_{\dot{W}^{s,\frac{p+1}{2}}} \, \mathrm{d}t' \notag
\end{align}
The term $\|v(t')\|_{\dot{W}^{s,\frac{p+1}{2}}}$ will be estimated by interpolating the estimates for $v(t')$ in $L^{p+1}$ and~$\dot{H}^1$. The standard tool to do so is the Gagliardo-Nirenberg inequality, see Theorem~\ref{1gagliardoNirenberg}. We obtain \[\|\nabla^{1-s} v\|_{L^{\frac{p+1}{2}}} \lesssim \|\nabla v\|_{L^2}^{1-\alpha} \|v\|_{L^{p+1}}^{\alpha} \lesssim E(t)^{\frac{1-\alpha}{2}+\frac{\alpha}{p+1}}\,,\] 
where $s$ satisfies the homogeneity conditions 
\[
\left \{
\begin{array}{c}
     s \leqslant \alpha  \\
     \frac{2}{p+1} - \frac{1-s}{3} = \frac{1-\alpha}{6} + \frac{\alpha}{p+1}
\end{array}
\right.\]
which gives $\alpha=s_p:=\frac{p-3}{p-1}$, so that \[E(t) \lesssim \|\nabla ^{s_p} z\|_{L^{\infty}_{T,x}}\int_0^t E(t') \, \mathrm{d}t'\] and the Grönwall lemma proves the non blow-up of $E(t)$, provided $\|\nabla ^{s_p} z\|_{L^{\infty}_{T,x}} < +\infty$ which is the case for initial data in $\mathcal{H}^{s_p+\varepsilon}$, thanks to probabilistic improvement of the Strichartz estimates that we will prove later.  

In our context it will not be possible to construct such a global strong solution $v$. However the strategy used in \cite{burq} applies: this is the strategy one uses to construct Leray solutions in the context of the Navier-Stokes equations, which consists in first finding approximate solutions $u_n=z_n+v_n$ that are global in time. Then the previous energy estimate provides uniform bounds for $v_n$ that allow strong compactness arguments in order to pass to the limit. 

\subsubsection{Yudovich-Wolibner argument} The Yudovich-Wolibner argument was first presented in the work of Wolibner in the context of the $2$-dimensional Euler equations, see~\cite{wolibner}. A similar argument was provided by Yudovich in the same context, see~\cite{yudovich}. We will recall this argument, in its simplest version. Note that this kind of argument has been widely used since, in particular in the study of~\eqref{1SLW}, $p=3, s=0$, see~\cite{burqTzvetkov}. 

Let us explain it in the context of the Schrödinger equation on $\mathbb{R}^2$:
\begin{equation}
    \label{1NLS}
    \tag{NLS}
    \left \{
    \begin{array}{ccc}
         i\partial_t u + \Delta u +u|u|^2 & =& 0\\
         u(0) & =& u_0 \in H^1(\mathbb{R}^2)\,. 
    \end{array}
    \right.
\end{equation}
Let $u_1,u_2 \in\mathcal{C}^0([0,T),H^1)$ be two solutions to \eqref{1NLS}. The following aims at proving uniqueness of solutions, that is $u_1=u_2$. In order to do so consider $E(t):=\|u_1(t)-u_2(t)\|^2_{L^2}$. Then a computation, using that $u_1, u_2$ solve~\eqref{1NLS} yields 

\[E'(t) \lesssim \int_{\mathbb{R}^2} |u_1(t)-u_2(t)|^2(|u_1(t)|^2+|u_2(t)|^2)\, \mathrm{d}x \,.\]
If the embedding $H^1(\mathbb{R}^2) \hookrightarrow L^{\infty}(\mathbb{R}^2)$ were true, then we would have $E'(t) \lesssim E(t)$ and would deduce that $E(t)=0$ for $t \in [0,T]$. Unfortunately there is no such continuous embedding, and we only have the \textit{Trudinger type} estimate~\cite{bahouriCheminDanchin} \[\|u\|_{L^p} \lesssim \sqrt{p}\|u\|_{H^1} \text{ for all } p>2\,.\] Using the previous inequality and the Hölder inequality gives
\begin{align}
    \label{1eqNonlin}
    E'(t) &\lesssim pE(t)^{1-\frac{1}{p}}
\end{align}
for all $p>1$. After integration by separation of variables this implies $E(t) \leqslant (Ct)^p$ for a constant $C>0$. Thus for a fixed $t< \frac{1}{C}$ and letting $p \to \infty$, we get $E(t)=0$. This argument can be iterated on time intervals $\left[\frac{n}{C},\frac{n+1}{C}\right]$ for $n \geqslant 0$ so that $E(t)=0$ for all $t\geqslant 0$. 

\begin{remark}
Another way to conclude is to optimize in $p$ in~\eqref{1eqNonlin} so that $E'(t) \lesssim -E(t) \log (E(t))$ and gives the same result. 
\end{remark}

As mentioned before our setting will be a little more complicated: we will need some bootstrap argument to conclude rather than this simple integration techniques. The method will be used to prove existence of global solutions in a limiting case rather than proving uniqueness, the framework being very similar. 

\subsubsection{Organization of the paper} Section~\ref{1estimeesProba} explains the the randomization procedure and recalls the probabilistic improvement for the Strichartz estimates. We then provide a generalization of these estimates in the context of Besov spaces, see Proposition~\ref{1propStrichartzBesov}.

Section~\ref{1sectionProof} is devoted to the proof of Theorem~\ref{1mainTheorem} in the case of the euclidean space $\mathbb{R}^3$. More precisely, sub-section~\ref{1sousSectionRegularized} proves existence of global solutions $u_n$ to approximate equations, sub-section~\ref{1sousSectionApriori} provides uniform bounds in $n$ for the nonlinear energies that will allow to use a compactness argument in sub-section~\ref{1sousSectionLimit}.

Section~\ref{1sectionTh2} provides the proof of Theorem~\ref{1th2} and its corollary. 

For reader's convenience some useful facts concerning Sobolev and Besov spaces are gathered in Appendix~\ref{1appendixA}.

\section*{Acknowledgement} 
The author warmly thanks Nicolas Burq and Isabelle Gallagher for careful advising during this work and suggesting this problem. The author is also grateful to the anonymous referee whose remarks helped improve the quality of this article. 

\section{Probabilistic estimates}\label{1estimeesProba} 

\subsection{The probabilistic setting} \label{1sousSectionProba} We first recall some standard notation in Littlewood-Paley analysis. 

Let $\mathcal{C}$ be the annulus $\{\xi \in \mathbb{R}^3,\; 3/4 \leqslant |\xi| \leqslant 8/3\}$, then there exists radial functions $\chi, \varphi$ taking values in $[0,1]$ belonging to $\mathcal{D}(B(0,4/3))$ and $\mathcal{D}(\mathcal{C})$ satisfying \[\chi (\xi) + \sum_{j \geqslant 0}\varphi(2^{-j}\xi)=1 \text{ for all } \xi \in \mathbb{R}^3\text{ and } \sum_{j \in \mathbb{Z}} \varphi (2^{-j}\xi)=1 \text{ for all } \xi \in \mathbb{R}^3 \setminus \{0\}\] and such that \[\begin{array}{ccc}
    |j-j'| \geqslant 2 & \Longrightarrow & \operatorname{supp} \varphi (2^{-j} \cdot) \cap \operatorname{supp} \varphi (2^{-j'} \cdot)= \varnothing \\
    j \geqslant 1 & \Longrightarrow & \operatorname{supp} \chi \cap \operatorname{supp} \varphi (2^{-j} \cdot) = \varnothing\,.
\end{array}\] 
We now define the nonhomogeneous Littlewood-Paley projectors: \[\Delta_j u := \left \{ \begin{array}{ccc}
    0 & \text{for} & j \leqslant -2\\
    \chi(D)u & \text{for} & j=-1\\
    \varphi(2^{-j}D)u & \text{for} & j\geqslant 0,
\end{array}\right.\]
where $\chi(D)$ (resp. $\varphi(2^{-j}D)$) denotes the Fourier multiplier of symbol $\chi$ (resp. $\varphi(2^{-j} \cdot)$). As a homogeneous Littlewood decomposition will be needed, we set $\dot{\Delta}_j := \varphi(2^{-j}D)$ for all $j \in \mathbb{Z}$. We set :  \[\mathbf{P}_j = \sum_{j' \leqslant j-1} \Delta_{j'} \text{ and } \dot{\mathbf{P}}_j=\chi(2^{-j}D).\]

In the case of the torus $\mathbb{T}^3$ we construct a similar decomposition, with a bump function $\varphi \in \mathcal{D}(B(0,2))$ such that $\varphi =1$ on $B(0,1)$. Let $(e_n)_{n \in \mathbb{Z}^3}$ be the hilbertian sequence of $L^2(\mathbb{T}^3)$ defined by $x \mapsto e_n(x)=e^{2i\pi n\cdot x}$. For a function $u=\sum_{n \in \mathbb{Z}^3} c_ne_n$, and for $j \geqslant 0$,  we set $\mathbf{P}_ju=\sum_{n \in \mathbb{Z}^3} \varphi(2^{-j}|n|)c_ne_n$ and $\Delta_ju:=\mathbf{P}_{j}u-\mathbf{P}_{j-1}u$, with the convention that $\mathbf{P}_{-1}=0$

An account of useful facts in Littlewood-Paley theory is given in Appendix~\ref{1appendixA}.

The randomization that is widely used in the context of the torus $\mathbb{T}^3$ or more generally a compact manifold is presented in \cite{burqTzvetkov2}.  Consider $(X_n)_{n \in \mathbb{Z}^3}$ a sequence of random variables defined on a probability space $(\Omega,\mathcal{A},\mathbb{P})$ which satisfy the following definition. 

\begin{definition}\label{1defAdmissible} Let $(\Omega, \mathcal{F}, \mathbb{P})$ be a probability space and $\left(X_n^{(i)}\right)_{n \in \mathbb{Z}^3 ; i=0,1}$ be a sequence of complex random variables defined on $\Omega$, satisfying the symmetry property $X_{-n}^{(i)}=\overline{X_n^{(i)}}$ and such that the random variables \[\left(X_0^{(i)},\operatorname{Re}\left(X_n^{(i)}\right), \operatorname{Im}\left(X_n^{(i)}\right)\right)_{n \in I ; i=0,1}\] where \[I=\left(\mathbb{Z}_+ \times \{0\}^{2}\right) \cup \left(\mathbb{Z} \times \mathbb{Z}_+ \times \{0\}\right) \cup \left(\mathbb{Z}^2 \times \mathbb{Z}_+ \right)\,,\] are independent; and that there exists a constant $c>0$ such that 
\begin{equation}
\label{1eqExp}
    \mathbb{E}\left[e^{\gamma X_n^{(i)}}\right] \leqslant e^{c\gamma ^2}
\end{equation}
for all $\gamma \in \mathbb{R}$ when $n=0$, for all $\gamma \in \mathbb{R}^2$ when $n \neq 0$.
\end{definition}

\begin{remark} Definition~\ref{1defAdmissible} immediately implies that the $X^{(i)}_0$, $i=0,1$ are real-valued and that the $X_n^{(i)}$ are of mean zero. Note that complex-valued Gaussian random variables, standard Bernoulli or random variables with compact support distributions satisfy the hypothesis of Definition~\ref{1defAdmissible}. 
\end{remark}

Every $u\in L^2(\mathbb{T}^3)$ can be written in the hilbertian basis $(e_n)_{n \in \mathbb{Z}}$ as $u=\sum_{n \in \mathbb{Z}^d} u_n e_n$ with $u_n$ the Fourier coefficients of $u$. We introduce the \textit{Fourier randomization} associated to a couple $(u_0,u_1)$ with the randomization map: 
\begin{equation}
    \label{1rand1}
    \Theta_{(u_0,u_1)} : \left \{
    \begin{array}{ccc}
     \Omega &  \longrightarrow & \{\text{ map from } \mathbb{T}^3 \text{ to } \mathbb{C}\}  \\
     \omega & \longmapsto & \left(\displaystyle\sum_{n \in \mathbb{Z}^3} X^{(0)}_n(\omega)u_{0,n}e_n, \displaystyle\sum_{n \in \mathbb{Z}^3} X^{(1)}_n(\omega)u_{1,n}e_n\right)\,.
    \end{array}
    \right. 
\end{equation}

There is a similar procedure in the euclidean setting. In this case the standard randomization setup is called the \textit{Wiener randomization} and randomizes frequencies annuli in a way that we explain. Note that a rough version of that randomization was introduced by Wiener in \cite{wiener}, and that the smooth version used here has been developed by Benyi-Oh-Pocovnicu in \cite{benyiOhPocovnicu}. Let $\psi \in \mathcal{D}(]-1,1[^3)$ with the symmetry property $\psi(-\xi)=\overline{\psi (\xi)}$ for all $\xi \in \mathbb{R}^3$ and satisfying the unit partition condition $\sum_{n \in \mathbb{Z}^3} \psi (\cdot -n) =1$, so that for every $u \in \mathcal{S}'$ there holds \[u=\sum_{n \in \mathbb{Z}^3} \psi (D-n)u\,.\]
One readily sees that
\begin{equation}
\label{1summationCondition}
\xi \longmapsto \sum_{n \in \mathbb{Z}^3} |\psi (\xi -n)|^2 \text{ is bounded.}
\end{equation}
The randomization of a couple $(u_0,u_1)$ is then defined with the randomization map 
\begin{equation}
    \label{1rand2}
    \Theta_{(u_0,u_1)} : \left \{
\begin{array}{ccc}
     \Omega &  \longrightarrow & \{\text{ map from } \mathbb{R}^3 \text{ to } \mathbb{C}\}  \\
     \omega & \longmapsto & \left(\displaystyle\sum_{n \in \mathbb{Z}^3} X_n^{(0)}(\omega) \psi (D-n)u_0,\displaystyle\sum_{n \in \mathbb{Z}^3} X_n^{(1)}(\omega) \psi (D-n)u_1 \right)\,.
\end{array}
\right. 
\end{equation}
In both cases (randomization in $\mathbb{T}^3$ or $\mathbb{R}^3$) we set $\mu_{(u_0,u_1)}(A) := \mathbb{P}\left(\Theta _{(u_0,u_1)}^{-1}(A)\right)$ for every measurable set $A$, and define \[\mathcal{M}^s(U):=\bigcap_{(u_0,u_1)\in H^s(U) \times H^{s-1}(U)} \mu_{(u_0,u_1)}\,,\] the set of all measures that we will work with. 

\begin{remark} In the following, when there is no possible confusion we will denote by $(u_0,u_1)$ the couple of random variables defined by the randomization maps \eqref{1rand1} and \eqref{1rand2}, rather than $(u_0^{\omega},u_1^{\omega})$ or any other notation. 
\end{remark}

These randomizations have been studied in \cite{burqTzvetkov2}, in which Burq-Tzvetkov proved the flolowing theorem. For a proof see Lemma B.1 in \cite{burqTzvetkov2} and also Lemma~2.2 in \cite{benyiOhPocovnicu}. 

\begin{theorem}[Non-smoothing effect of the randomization setup] Let $\mu = \mu_{(u_0,u_1)} \in \mathcal{M}^s(U)$, $U=\mathbb{R}^3$ or $\mathbb{T}^3$. We have the following: 
\begin{enumerate}[label=(\textit{\roman*})]
    \item The measure $\mu$ is supported by $\mathcal{H}^s(U)$.
    \item For $s'>s$ if $(u_0,u_1) \notin \mathcal{H}^{s'}(U)$ then $\mu(\mathcal{H}^{s'}(U))=0$. 
    \item In the case of the torus, if all the Fourier coefficients of $u_0$ and $u_1$ are different from zero and the support of the distributions of the $\left(X_0^{(i)},\operatorname{Re}(X_n^{(i)}),\operatorname{Im}(X_n^{(i)})\right)_{n\in I \; ; \; i=0,1}$ are $\mathbb{R}$ then the support of $\mu$ is exactly $\mathcal{H}^s(U)$. 
\end{enumerate}
\end{theorem}

This theorem proves that there is no gain in regularity by randomization. However a gain in integrability is known, see Theorem~\ref{1thKolmogorov}, and is responsible for the improvement in Strichartz inequalities and thus the local wellposedness and global well-posedness theory in dispersive equations.  

\subsection{Probabilistic semigroup estimates}

The starting point of every probabilistic improvement of the Strichartz estimates is the following well-known theorem in the theory of random Fourier series. For a proof see Lemma~3.1 in \cite{burqTzvetkov2} and Lemma~2.1 in \cite{ohPocovnicu}. 

\begin{theorem}[Kolmogorov-Paley-Zygmund]\label{1thKolmogorov} Let $(X_n)_{n \in \mathbb{Z}}$ be a sequence of independent and identically distributed real-valued random variables such that there exists a constant $c>0$ such that for every $\gamma \in \mathbb{R}$, the inequality \eqref{1eqExp} holds. Let $(a_n)_{n \in \mathbb{Z}} \in \ell^2(\mathbb{Z})$ be a complex valued sequence with the symmetry property $a_{-n} = \overline{a_n}$ for every integer $n$ (resp. a real-valued sequence $(a_n)_{n \in \mathbb{N}} \in \ell^{2}(\mathbb{N})$). For $q \in [1,\infty)$ one has: 
    \begin{equation}\label{1kolmogorovPaley}
    \left \| \sum_{n \in \mathbb{Z}}  a_n X_n \right \|_{L^q} \lesssim  \sqrt{q} \|(a_n)_{n \in \mathbb{Z}}\|_{\ell^2}.
    \end{equation}
\end{theorem}

We turn to the probabilistic improvement of Strichartz estimates and introduce semi-groups associated to the linear wave equation. Let $s \geqslant 0$ and $(u_0,u_1) \in \mathcal{H}^s(U)$. We set:
\[z(t):=S(t)(u_0,u_1):= \cos(t |\nabla|)u_0 + \frac{\sin(t|\nabla|)}{|\nabla|}u_1,\]
which is a solution to \eqref{1LW}. One of the key features of the linear wave equation is that ``two time derivatives equal two space derivatives'' so that one expects that ``one time derivative equals one space derivative'' which can be turned more rigorously writing $\partial _t z(t)=\langle \nabla \rangle \tilde{z}(t)$ where
\[
\tilde{z}(t):=\tilde{S}(t)(u_0,u_1) := - \frac{|\nabla|\sin(t |\nabla|)}{\langle \nabla \rangle}u_0 + \frac{\cos(t|\nabla|)}{\langle \nabla \rangle}u_1.
\]
For our purposes we will need a smooth version of both $z(t)$ and $\tilde{z} (t)$, namely
\begin{equation}
    \label{1smoothSemiGroup}
    z_n(t):=S_n(t)(u_0,u_1):=\mathbf{P}_n S(t)(u_0,u_1) \text{  and  } \tilde{z}_n(t):=\tilde{S}_n(t)(u_0,u_1):=\mathbf{P}_n \tilde{S}(t)(u_0,u_1)
\end{equation}
for every integer $n \geqslant 1$. We recall the probabilistic Strichartz estimates, proven in \cite{pocovnicu,ohPocovnicu2} for~\eqref{1probabilisticStrichartzLebesgue} and \cite{ohPocovnicu,ohPocovnicu2} for \eqref{1probabilisticStrichartzLebesgue2}.

\begin{proposition}[Probabilistic Strichartz estimates for the wave operator] \label{1propStrichartz} Let $U$ standing for either $\mathbb{T}^3$ or $\mathbb{R}^3$. Let $(u_0,u_1) \in \mathcal{H}^s(U)$ and still write $(u_0,u_1)$ its randomization (Fourier randomization procedure or Wiener randomization procedure). Let $z^*$ stand for either $z,\tilde{z},z_n$ or~$\tilde{z}_n$. For any $q_1\in[1,\infty)$ and $q_2 \in [2,\infty]$: 
\begin{equation}
\label{1probabilisticStrichartzLebesgue}    
\mathbb{P}(\|z^*\|_{L^{q_1}((0,T),W^{s,q_2}(\Omega))} > \lambda) \lesssim_{\varepsilon}\exp \left( - \frac{c\lambda ^2}{\max \left\{T^{\frac{2}{q_1}},T^{2+\frac{2}{q_1}}\right\}\|(u_0,u_1)\|_{\mathcal{H}^{s+\varepsilon}}^2}\right)
\end{equation}
with $\varepsilon = 0$ if $q_2 < \infty$ and $\varepsilon >0$ arbitrarily small otherwise. \\ For any $q_2 \in [2,\infty]$ and arbitrarily small $\varepsilon >0$:
\begin{equation}
\label{1probabilisticStrichartzLebesgue2}    
\mathbb{P}(\|z^*\|_{L^{\infty}((0,T),W^{s,q_2}(\Omega))} > \lambda) \lesssim_{\varepsilon} (1+T)\exp \left( - \frac{c\lambda ^2}{\max\{1,T^2\}\|(u_0,u_1)\|_{\mathcal{H}^{s+\varepsilon}}^2}\right)\,.
\end{equation}
\end{proposition}

For our purposes we will need a counterpart of Proposition~\ref{1propStrichartz} in the context of Besov spaces:

\begin{proposition}[Besov norm probabilistic Strichartz estimates]\label{1propStrichartzBesov} Let $(u_0,u_1) \in \mathcal{H}^s(U)$ where $U$ stands for either $\mathbb{T}^3$ or $\mathbb{R}^3$ and still write $(u_0,u_1)$ its randomization (Fourier randomization procedure or Wiener randomization procedure). Let $z^*$ stand for either $z,\tilde{z},z_n$ or $\tilde{z}_n$. For any $q_1 \in [1,\infty)$, $q_2 \in [2,\infty]$ and $r \in [1,\infty]$: 
\begin{equation}\label{1probaStrichartzBesov}
    \mathbb{P}\left( \|z^*\|_{L^{q_1}((0,T),B_{q_2,r}^s)} > \lambda \right) \lesssim_{\varepsilon} \exp \left( - c \frac{\lambda ^2}{\max\{T^{\frac{2}{q_1}},T^{2+\frac{2}{q_1}}\}\|(u_0,u_1)\|_{\mathcal{H}^{s+\varepsilon}}^2}\right)
\end{equation}
with $\varepsilon = 0$ if $q_2 < \infty$ and $r \geqslant 2$; $\varepsilon >0$ otherwise. \\
For any $q_2\in [2,\infty]$ and $r \in [1,\infty]$: 
\begin{equation}\label{1probaStrichartzBesov2}
    \mathbb{P}\left( \|z^*\|_{L^{\infty}((0,T),B_{q_2,r}^s)} > \lambda \right) \lesssim_{\varepsilon} \exp \left( - c(\varepsilon) \frac{\lambda ^2}{\langle T \rangle ^{2\varepsilon}\|(u_0,u_1)\|_{\mathcal{H}^{s+\varepsilon}}^2}\right)
\end{equation}
with $\varepsilon >0$.
\end{proposition}

\begin{remark} Note that the estimate \eqref{1probaStrichartzBesov2} differs slightly from \eqref{1probabilisticStrichartzLebesgue2}. The proof presented here will indeed differ from the one in \cite{ohPocovnicu} which appears in the context of Lebesgue spaces and relies on a series representation for $z(t)$, a method that we decided not to use and present an alernative method. However, by applying the method of Oh-Pocovnicu one can prove a similar estimate.  
\end{remark}

\begin{proof}The proof follows closely the one in \cite{ohPocovnicu,ohPocovnicu2,pocovnicu} as only the parameter $r$ has been added in the analysis. Nonetheless the proof is given in quite extensive details for reader's convenience. We will only give the proof in the case $U=\mathbb{R}^3$ since the computations are almost the same in $\mathbb{T}^3$. It is indeed only the randomization setup which differs but in both cases they satisfy the same smoothing properties, see Appendix~\ref{1appendixA}. We will also assume that $z^*(t)=z(t)$, other cases could be treated in the exact same way.

\textbf{\textit{Step 1.}} Assume $s=0$, $q_1< \infty$ and $q_2< \infty$. Without loss of generality we can assume that $r < \infty$ thanks to the inequality $\|z\|_{L^{q_1}((0,T),B^{0}_{q_2,\infty})} \lesssim \|z\|_{L^{q_1}((0,T),L^{q_2})}$ and~\eqref{1probabilisticStrichartzLebesgue}. Assume first that $r \geqslant 2$. We will prove that for $p \geqslant \max \{q_1,q_2,r\}$ one has
\begin{equation}\label{1eqNormes}
    \left\| z \right\|_{L^{p}(\Omega, L^{q_1}((0,T),B^0_{q_2,r}))} \lesssim \sqrt{p} T^{1+\frac{1}{q_1}} \|(u_0,u_1)\|_{\mathcal{H}^0}. 
\end{equation}

Assume that \eqref{1eqNormes} is proved, then the Markov inequality (with the function $\lambda \mapsto \lambda ^p$) gives \[ \mathbb{P} \left( \|z\|_{L^{q_1}((0,T),B_{q_2,r}^0)} > \lambda \right) \leqslant \lambda^{-p}\left\| z \right\|_{L^{p}(\Omega, L^{q_1}((0,T),B^0_{q_2,r}))}^p\] and minimizing in $p$ yields \eqref{1probaStrichartzBesov}. It is indeed the case when the optimizing $p$ is such that $p \geqslant \max\{q_1,q_2,r\}$, otherwise just take $C$ large enough to ensure $Ce^{-\max\{q_1,q_2,r\}} \geqslant 1$ and write \[\mathbb{P}\left( \|z\|_{L^{\infty}((0,T),B_{q_2,r}^0)} > \lambda \right) \leqslant 1 \leqslant Ce^{-\max\{q_1,q_2,r\}} \leqslant Ce^{-p}\,,\] which ends the proof of \eqref{1probaStrichartzBesov}.

The proof now reduces to the one of \eqref{1eqNormes}. Since $z(t)=\cos(t|\nabla|)u_0+\frac{\sin(t|\nabla|)}{|\nabla|}u_1$ we assume that $z(t)=\frac{\sin(t|\nabla|)}{|\nabla|}u_1$ and only estimate this term (the other is even simpler to handle and we omit the details). Set $p \geqslant \max\{q_1,q_2,r\}$, use the integral Minkowski inequality and Theorem~\ref{1thKolmogorov}:

\begin{align*}
\lefteqn{A:=\left\| \frac{\sin(t|\nabla|)}{|\nabla|}u_1 \right\|_{L^{p}(\Omega, L^{q_1}((0,T),B^0_{q_2,r}))} \leqslant  \left\| \left\|  \Delta _j \frac{\sin(t|\nabla|)}{|\nabla|}u_1 \right\|_{L^{p}(\Omega)} \right\|_{L^{q_1}((0,T),\ell_j ^r(\mathbb{N},L^{q_2}))} } \\
& & \lesssim \sqrt{p} \left \| \left\|\psi (D-n) \frac{\sin(t|\nabla|)}{|\nabla|} \Delta _j u_1\right\|_{\ell ^2_n(\mathbb{Z})}\right\|_{L^{q_1}((0,T),\ell_j ^r(\mathbb{N},L^{q_2}))} \,. \notag
\end{align*}
As $q_2 \geqslant 2$, the Minkowski inequality followed by the Bernstein inequality on $\psi (D-n)$ (see~\ref{1coroBernstein}) imply
\begin{align*}
    \lefteqn{A \lesssim \sqrt{p} \left \| \left\|\psi (D-n) \frac{\sin(t|\nabla|)}{|\nabla|} \Delta _j u_1\right\|_{L^{q_2}}\right\|_{L^{q_1}((0,T),\ell_j ^r(\mathbb{N},\ell_n^2(\mathbb{Z})))}} &\\
    & \lesssim \sqrt{p} \left \| \left\|\psi (D-n) \frac{\sin(t|\nabla|)}{|\nabla|} \Delta _j u_1\right\|_{L^2}\right\|_{L^{q_1}((0,T),\ell_j ^r(\mathbb{N},\ell_n^2(\mathbb{Z})))} \\
    & \lesssim \sqrt{p}\left\| \left(t^2\left\|\frac{\psi(D)\sin(t |\nabla|)}{t|\nabla|} \Delta _ju_1\right\|_{L^2}^2 +\sum_{|n|\geqslant 1} \left\| \frac{\psi(D-n)\sin(t |\nabla|)}{|\nabla|} \Delta _j u_1\right\|_{L^2}^2\right)^{1/2} \right\|_{L^{q_1}((0,T),\ell_j^r(\mathbb{N}))}\,.
\end{align*}

The use of the elementary inequalities $\left \vert \frac{\sin x}{x} \right \vert \leqslant 1$ and $|\sin (x)| \leqslant 1$ for all $x \in \mathbb{R}$ along with the Bernstein inequality give: 
\begin{align*}
A & \lesssim \sqrt{p}\left\| \left( t^2\left\|\psi(D)\Delta _ju_1\right\|_{L^2}^2 + \sum_{|n|\geqslant 1} \left\| \frac{\psi(D-n)}{|\nabla|} \Delta _j u_1\right\|_{L^2}^2 \right)^{1/2} \right\|_{L^{q_1}((0,T),\ell_j^r(\mathbb{N}))} \\ 
& \lesssim \sqrt{p} \max\left\{T^{\frac{1}{q_1}},T^{2+\frac{1}{q_1}}\right\} \left \| \left( \sum_{n \in \mathbb{Z}} \left\| \psi(D-n) \Delta _j (\langle \nabla \rangle ^{-1}u_1)\right\|_{L^2}^2 \right)^{1/2}\right \|_{\ell ^r_j(\mathbb{N})}\,, 
\end{align*}
and using \eqref{1summationCondition} as well as the definition of Besov spaces this implies \[\left\| \frac{\sin(t|\nabla|)}{|\nabla|}u_1 \right\|_{L^{p}(\Omega, L^{q_1}((0,T),B^0_{q_2,r}))} \lesssim \sqrt{p}\max\left\{T^{\frac{1}{q_1}},T^{2+\frac{1}{q_1}}\right\} \|u_1\|_{B^{-1}_{q_2,r}}\,.\]

When $r \geqslant 2$ the conclusion follows from the fact that $H^{-1} \simeq B^{-1}_{2,2} \hookrightarrow B^{-1}_{2,r}$, see Theorem~\ref{1sobolev}. For $r<2$ it follows from the fact that for all $\varepsilon >0$ arbitrarly small, $H^{-1+\varepsilon} \hookrightarrow B^{-1+\varepsilon}_{2,2} \hookrightarrow B^{-1}_{2,r}$ which explains the loss of derivatives. See also Theorem~\ref{1sobolev}. 

\textbf{\textit{Step 2.}} The case where $s>0$, $q_1< \infty$ and $q_2< \infty$ is inferred by the case $s=0$ using that $\langle \nabla\rangle^s $ commutes with semi-groups $S(t),\tilde{S}(t)$.

\textbf{\textit{Step 3.}} The case where $s>0$, $q_1 < \infty$ and $q_2=\infty$ follows from Sobolev-Besov continuous embeddings given in Theorem~\ref{1sobolev} in the usual manner: for $q_2>\frac{3}{\varepsilon}$, $B^{\varepsilon}_{q_2,r}\hookrightarrow B^{\varepsilon}_{\infty,r}$ so that the conclusion follows with an $\varepsilon$ loss of derivatives. 

\textbf{\textit{Step 4.}} Assume that $(s,q_1)=(0,\infty)$ and $r \in[2,\infty )$ which is the last case we need to address. Other cases will follow from the use of Step 2 and Step 3. 

Let $q$ large enough such that $\varepsilon q > 1$, for example $q=\frac{2}{\varepsilon}$, which ensures that the embedding $W^{\varepsilon, q}(\mathbb{R}) \hookrightarrow L^{\infty}(\mathbb{R})$ is continuous. Then \[\|z\|_{L^{\infty}((0,T),B^0_{q_2,r})} \lesssim_{\varepsilon} \langle T \rangle^{\varepsilon}\left\|\langle t \rangle ^{-\varepsilon}z \right\|_{W^{\varepsilon,q}(\mathbb{R},B^0_{q_2,r})}\,.\] 

Note that similarly as in Step 1-3 (see \cite{burqTzvetkov2}, Proposition~A.5 for details) that with $\delta :=\frac{2}{q_1}$ one has \begin{equation}
    \label{1probaStrichartzBesov3}
    \mathbb{P}\left( \left\|\langle t \rangle^{-\delta} z \right\|_{L^{q_1}(\mathbb{R},B_{q_2,r}^s)} > \lambda \right) \lesssim_{\varepsilon} \exp \left( - c \frac{\lambda ^2}{\|(u_0,u_1)\|_{\mathcal{H}^{s+\varepsilon}}^2}\right)
\end{equation}
with $\varepsilon = 0$ if $q_2 < \infty$ and $r \geqslant 2$; $\varepsilon >0$ otherwise.

Now, using the representation of $z$ in terms of exponentials rather than trigonometric we can assume without loss of generality that $z(t)=e^{it|\nabla|}\phi$ where $\phi \in B^{\varepsilon}_{q_2,r}$. For clarity reasons set $\chi (t) := \langle t \rangle^{-\varepsilon}$ and observe that in order to conclude we only need to prove
\begin{equation}
    \label{1egalite}
    \|\chi z\|_{W^{\varepsilon,q}(\mathbb{R},B^0_{q_2,r})} \lesssim_{\varepsilon} \|\chi z\|_{L^{q}(\mathbb{R},B^{\varepsilon}_{q_2,r})}\,,
\end{equation}
since conditionally to~\eqref{1egalite}, the estimate \eqref{1probaStrichartzBesov3} applies to the latter norm. 

In order to do so, remark that $\|\chi z\|_{W^{\varepsilon,q}(\mathbb{R},B^0_{q_2,r})} \lesssim \| \langle D _t \rangle ^{\varepsilon}( \chi z)\|_{L^{q}(\mathbb{R},B^0_{q_2,r})}$. Next we write \begin{align*}
    \langle D_t \rangle ^{\varepsilon} \chi &= \chi \langle D_t \rangle ^{\varepsilon} + \left[\langle D_t \rangle ^{\varepsilon}, \chi\right] \\
    & = \left( 1 + \left[\langle D_t \rangle ^{\varepsilon}, \chi\right]\langle D_t \rangle ^{-\varepsilon}\chi ^{-1} \right) \chi \langle D_t \rangle ^{\varepsilon}  \\
    & := A \chi \langle D_t \rangle ^{\varepsilon}.
\end{align*}
Now remark that $\langle D_t \rangle ^{\varepsilon} \chi$ is a pseudo-differential operator with symbol in $S^{\varepsilon}$ and $\chi, \chi ^{-1}$ are pseudo-differential operators of order zero. The standard pseudo-differential calculus now shows that $\left[\langle D_t \rangle ^{\varepsilon}, \chi\right]$ is of order $\varepsilon -1<0$ and thus $A$ is of order zero. Such operators are known to be continuous on $L^q$, see \cite{hormander} for instance. This yields 
\begin{equation}
    \label{1almostFinished}
    \|\chi z\|_{W^{\varepsilon,q}(\mathbb{R},B^0_{q_2,r})} \lesssim_{\varepsilon} \|\chi  \langle D _t\rangle ^{\varepsilon }z\|_{L^{q}(\mathbb{R},B^{0}_{q_2,r})}\,.
\end{equation}
To finish the proof of~\eqref{1egalite} remark that \[\langle D_t \rangle^{\varepsilon} z = \langle \nabla \rangle^{\varepsilon} z \text{ in } \mathcal{S}'(\mathbb{R} \times \mathbb{R}^3)\,\] thus when plugged in \eqref{1almostFinished} this implies~\eqref{1egalite} and concludes the proof. 
\end{proof}

\section{Proof of Theorem~\ref{1mainTheorem}}\label{1sectionProof}

We provide the proof of Theorem~\ref{1mainTheorem} in the case of $U=\mathbb{R}^3$. The case of the torus $\mathbb{T}^3$ is very similar as the Littlewood-Paley analysis works the same. For other adaptations to the case of $\mathbb{T}^3$ see proof of the probabilistic well-posedness in the subcritical regime $3<p<5$ in~\cite{sunXia} and the proof in the critical regime $p=5$ in~\cite{ohPocovnicu2}.

\subsection{Global strong solutions for the regularized system}\label{1sousSectionRegularized} 
In order to derive \textit{a priori} energy estimates for \eqref{1SLW} we first construct global strong solutions for approximate equations. In order to do so we use a smooth truncation in frequencies, which will prove helpful in the following. 

Set $f(x)=|x|^{p-1}x$ and consider the regularized equation for $n \geqslant 1$:
\begin{equation}
\tag{rSLW$_p^n$}
\label{1rSLW}
    \left \{
    \begin{array}{cc}
         \partial_t^2 u_n - \mathbf{P}_n \Delta u_n + \mathbf{P}_nf(u_n)=0, \\
         (u_n(0), \partial _t u_n (0))=(\mathbf{P}_n u_0, \mathbf{P}_n u_1) \in \mathcal{H}^1(\mathbb{R}^3)\,.
    \end{array}
    \right.
\end{equation} 
We prove existence of a unique global solution $(u_n,\partial_t u_n)$ in the space 
\[X_n:=L_n^2(\mathbb{R}^3) \times L_n^2(\mathbb{R}^3) \text{ where } L_n^2(\mathbb{R}^3) := \{f \in L^2(\mathbb{R}^3), \; \mathbf{P}_nf=f\}.\]
Endowed with the norm $\|(u,v)\|_{X_n}:=\|u\|_{L^2} + \|v\|_{L^2}$, $X_n$ is a Banach space.

\begin{proposition}[Study of \eqref{1rSLW}] 
\begin{sloppypar} There exists unique global strong solutions $(u_n)_{n \geqslant 1}$to the equations \eqref{1rSLW} that belong to the spaces $X_n$. Moreover $u_n \in H^1 \cap L^{p+1}$ and for every $n \geqslant 1$, every $t \in \mathbb{R}$:
\end{sloppypar} 
\begin{align}
\label{1smoothEnergy}
     E_{\operatorname{reg}}(u_{n}(t),\partial _t u_n(t)) &:= \int_{\Omega} \left(  \frac{|\partial _t u_{n}(t)|^2}{2} + \frac{|\nabla u_{n}(t)|^2}{2} + \frac{|u_n(t)|^{p+1}}{p+1}\right) \, \mathrm{d}x \\
     &=E_{\operatorname{reg}}(\mathbf{P}_nu_0,\mathbf{P}_nu_1). \notag
\end{align}
\end{proposition}

\begin{proof} The proof is standard as local existence and uniqueness is achieved via the Picard-Lindelöf theorem, and the global existence will result from energy conservation. The remaining of this proof provides details of this classical scheme. Before starting the proof, remark that by the time reversibility of \eqref{1rSLW} it is sufficient to show existence and uniqueness of global solutions on the time interval $\mathbb{R}_+$. 

We start by proving that the equation \eqref{1rSLW} is locally well-posed in $\mathcal{C}^1(\mathbb{R}_+,X_n)$. It is a consequence of the Picard-Lindelöf theorem once we have written \eqref{1rSLW} in the form \[\frac{\mathrm{d}}{\mathrm{d}t} U_n (t) = F_n(U_n(t)), \text{ with } U_n(t):=(u_n(t),\partial _t u_n(t)) \text{ and } F_n(u,v):=(v,\mathbf{P}_n \Delta u - \mathbf{P}_nf(u)).\]

In order to be applied, the Picard-Lindelöf theorem requires the map $F_n$ to be locally Lipschitz on $X_n$. As $u \mapsto \mathbf{P}_n \Delta u$ is linear and continuous from $L_n^2$ into itself, it is locally Lipschitz. Observe that for $u \in L^2_n$, the Bernstein inequality proves that $u \in L^{\infty}$ and more precisely observe that $\|u\|_{L^{\infty}} \lesssim 2^{\frac{3n}{2}}\|u\|_{L^2}$ so that $F_n$ is well-defined from $L_n^2$ into itself. Finally let $(u,v),(u',v') \in X_n$ satisfying $\|(u,v)\|_{X_n},\|(u',v')\|_{X_n} \leqslant R$ and compute:
\begin{align*}
    \|F_n(u,v)-F_n(u',v')\|_{X_n} &\leqslant \|v-v'\|_{L^2} + \|\mathbf{P}_n \Delta (u-u')\|_{L^2} + \|\mathbf{P}_n(f(u)-f(u'))\|_{L^2}\\
& \lesssim_n \|v-v'\|_{L^2} + \|u-u'\|_{L^2} + \|f(u)-f(u')\|_{L^2} \\
& \lesssim_n \|(u,v)-(u',v')\|_{X_n}.
\end{align*}
Notice that in the last inequality we used that \[\||u|^{p-1}u-|v|^{p-1}v\|_{L^2} \lesssim_n (\|u\|^{p-1}_{L^{\infty}}+\|v\|^{p-1}_{L^{\infty}})\|u-v\|_{L^2} \lesssim _n R^{p-1} \|u-v\|_{L^2}.\] The Picard-Lindelöf theorem applies and gives rise to unique solutions defined on maximal time intervals that we denote $[0,T_n)$. These solutions belong to $\mathcal{C}^1([0,T_n),X_n)$. 

In order to derive the energy estimates, if we prove that $u_{n}$ has regularity $\mathcal{C}^2$ in both space and time, then it is sufficient to multiply \eqref{1rSLW} by $\partial _t u_n(t)$, integrate by parts in space and use the fact that the operator $\mathbf{P}_n$ is symmetric in $L^2$ to obtain \eqref{1smoothEnergy}. Time regularity is granted from the regularity given by the Picard-Lindelöf theorem. For space regularity observe that since $u_n\in L_n^2$, the derivation is a continuous mapping of $L^2_n$, thus $u_n \in H^k$ for all $k \geqslant 0$. The Sobolev embedding theorem proves the required smoothness in space for $u_n$.

Let $n \geqslant 1$. We prove that $T_n= + \infty$. Remark that the energy equality \eqref{1smoothEnergy} proves that \[\sup_{t \in [0,T_n)} \left(\|\nabla u_n(t)\|_{L^2} + \|\partial _t u_n(t)\|_{L^2}\right) < \infty.\] Assume that $T_n<\infty$. Then for $t \in(0,T_n)$ \[\|u_n(t)\|_{L^2} \leqslant \|u_0\|_{L^2} + \int_0^{t}\|\partial _t u_n(s)\|_{L^2} \, \mathrm{d}s \leqslant \|u_0\|_{L^2} + T_n \sup_{t \in (0,T_n)} \|\partial _t u(t)\|_{L^2}\] which yields $\sup_{t \in [0,T_n)} \|u_n(t)\|_{L^2} < \infty$. 

Note that the Bernstein inequality implies $\|\mathbf{P}_n \Delta (u_n)\|_{L^2} \lesssim _n \|u_n\|_{L^2}$ and  $\|\mathbf{P}_nf(u_n(t))\|_{L^2} \lesssim _n \|u_n(t)\|_{L^{2p}}^p \lesssim _n \|u_n(t)\|_{L^2}^p$ so that: \[\sup_{t \in [0,T_n)} \left\| \frac{\mathrm{d}}{\mathrm{d}t} U_n(t)\right\|_{X_n} = \sup_{t \in [0,T_n)} \left\{\|\partial_t u_n(t)\|_{L^2}+\|\mathbf{P}_n \Delta u_n(t) - \mathbf{P}_n f(u_n(t))\|_{L^2} \right\}< \infty\] which allow to construct a continuation for $U_n$ at $t=T_n$ and contradicts the maximality of $T_n$. Finally $T_n= \infty$.
\end{proof}

We now write $u_n(t)=z_n(t)+v_n(t)$ where $z_n$ has been introduced in \eqref{1smoothSemiGroup}, with initial data $(z_n(0),\partial _t z_n(0))=\mathbf{P}_n(u_0,u_1)$, and $v_n$ satisfying 
\begin{equation}
\label{1regNonlinear}
    \left \{ 
    \begin{array}{c}
         \partial _t^2v_n(t) - \mathbf{P}_n\Delta v_n(t) + \mathbf{P}_n \left((z_n(t)+v_n(t))|z_n(t)+v_n(t)|^{p-1}\right)=0, \\
         (v_n(0),\partial _t v_n(0))=(0,0). 
    \end{array}
    \right.
\end{equation}

\subsection{A priori estimates for the regularized system}\label{1sousSectionApriori} In order to pass to the limit $n \to \infty$, one needs uniform estimates for \eqref{1rSLW}. As we expect the linear part to be handled in a simple way we may focus on the nonlinear part $v_n$, satisfying \eqref{1regNonlinear}, and introduce its nonlinear energy
\[E_n(t) := \int_{\mathbb{R}^3} \left( \frac{|\partial _t v_{n}(t)|^2}{2} + \frac{|\nabla v_{n}(t)|^2}{2} + \frac{|v_n(t)|^{p+1}}{p+1}\right) \, \mathrm{d}x\,.\]
Let $s_p:=\frac{p-3}{p-1}$. In this subsection we prove the following uniform bound. 
\begin{proposition}[Probabilistic a priori estimates]\label{1propProbabilistic} Let $s>s_p$, and $p>5$. Let $T>0$ and $\eta \in (0,1)$. There exists a measurable set $\Omega_{T,\eta} \subset \Omega$ and a constant $C(T, \eta, \|(u_0,u_1)\|_{\mathcal{H}^s})$ which depends only on $T,\eta, \|(u_0,u_1)\|_{\mathcal{H}^s}$, such that:
\begin{enumerate}[label=(\textit{\roman*})]
    \item $\mathbb{P}(\Omega_{T,\eta}) \geqslant 1- \eta$.
    \item For every $\omega \in \Omega_{T,\eta}$, if the initial data for $u_n$ is attached to $\omega$ \textit{via} the randomization map of $(u_0,u_1)$ then: \[\sup_{n \geqslant 0} \sup_{t \in(-T,T)} E_n(t) \leqslant C(T, \eta, \|(u_0,u_1)\|_{\mathcal{H}^s}).\]
\end{enumerate}
\end{proposition}

The cornerstone of the proof of Proposition~\ref{1propProbabilistic}, which will allow to close the energy estimates in the Grönwall argument is the following. We introduce $\alpha_p := \lceil \frac{p-3}{2} \rceil$ and recall that $f(x)=|x|^{p-1}x$.

\begin{lemma} \label{1lemmeTechnique} For every $1 \leqslant k \leqslant \alpha_p$ one has: 
\begin{equation}\label{1estimeeTechnique}
\left \vert \int_{\mathbb{R}^3} f^{(k-1)}(v_n(t)) z_n(t)^{k-1} \langle \nabla \rangle \tilde{z}_n(t) \, \mathrm{d}x \right \vert \lesssim g\left(\|z_n\|_{L^{\infty}((-T,T),X)},\|\tilde z_n\|_{L^{\infty}((-T,T),Y)}\right)\left(1+E_n(t)\right).
\end{equation}
where $g$ is a polynomial with positive coefficients, 
\[X:=L^{\infty} \cap L^{\frac{p+1}{2}} \cap B^{1-s_p}_{q_2,1} \cap B^{1-s_p}_{q_{\alpha_p},1} \, \text{ and }\, Y:=L^{\infty} \cap B^{s_p}_{\infty,1}\] where for $2 \leqslant k \leqslant \alpha_p$, $q_k$ being defined by $\frac{1}{q_k}+\frac{p-k+1}{p+1}=1$.
\end{lemma} 

\begin{remark} Observe that for $p>5$ we have $1-s_p<s_p$. 
\end{remark}

For exposition reasons we postpone its proof, and prove Proposition~\ref{1propProbabilistic} assuming Lemma~\ref{1lemmeTechnique}. 

\begin{proof}[Proof of Proposition~\ref{1propProbabilistic}] Once again, by time reversibility, we will prove an \textit{a priori} estimate on $[0,T)$ rather than $(-T,T)$. We will first find a large measure set allowing to prove the desired estimates. Note that the forthcoming constraints in the definition of $\Omega_{T,\eta}$ are designed to control all the terms requiring bounds for the linear parts $z_n$ or $\tilde{z}_n$ that will be proven in the following. Let $\lambda >0$, which will be chosen later. Set 
\begin{align}
\label{1eqProbaSet}
    \Omega_{T,\eta} := \left\{ \right.&\|z_n\|_{L^{\infty}_TL^{p+1}} + \|z_n\|_{L^{2p}_T L^{2p}} + \|z_n\|_{L^{\infty}_T L^{r_p(\alpha_p+1)}}^{\alpha_p+1} \\ &\left.+g\left(\|z_n\|_{L^{\infty}((0,T),X)},\|\tilde z_n\|_{L^{\infty}((0,T),Y)}\right) \leqslant \lambda, \text{for all } n\geqslant 1 \right\},\notag
\end{align}
where $r_p$ is defined by $\frac{1}{2} + \frac{p-\alpha_p-1}{p+1} + \frac{1}{r_p}=1$.

For $\lambda$ large enough, depending only on the initial data, $T$ and~$\eta$, we can assume $\mathbb{P}(\Omega_{T,\eta}) \geqslant 1- \eta$ thanks to Proposition~\ref{1propStrichartz} and Proposition~\ref{1propStrichartzBesov}, as soon as $s>s_p$. The fact that $\lambda$ does not depend on $n$ comes from the inequality $\|\mathbf{P}_n(u_0,u_1)\|_{\mathcal{H}^s} \leqslant \|(u_0,u_1)\|_{\mathcal{H}^s}$.

Let $\omega \in \Omega_{T,\eta}$. From now on the following estimates will be deterministic as we have fixed the initial data (attached to $\omega$ \textit{via} the randomization map). We will carry out the computations for $s=s_p$. 

Recall that $\mathbf{P}_n$ is a symmetric operator in $L^2(\mathbb{R}^3)$ and that $v_n$ is smooth. This allows one to compute $\frac{\mathrm{d}}{\mathrm{d}t} E_n(t)$ and obtain: 
\begin{align}
\label{1eqDerivee}
  \frac{\mathrm{d}}{\mathrm{d}t} E_n(t) &= \int_{\mathbb{R}^3} \partial_t v_{n}(t) \left((\partial_t^2 v_{n}(t)- \mathbb{P}\Delta v_n(t) + \mathbf{P}_n(v_n(t)|v_n(t)|^{p-1})\right) \, \mathrm{d}x\\
  &= - \int_{\mathbb{R}^3} \partial_t v_n(t) \left( f(z_n(t)+v_n(t))-f(v_n(t)) \right) \, \mathrm{d}x. \notag
\end{align}
Next we expand the nonlinearity $f$ at the point $v(t,x)$ using the Taylor formula with integral remainder up to the order $\alpha_p=\lceil \frac{p-3}{2} \rceil$. For convenience we drop the $t,x$ references and recall that for $k \geqslant 0$,
\[
f^{(k)}(x) = \left \{ \begin{array}{ccc}
    C_{p,k}x|x|^{p-k-1} & \text{for}& k \text{ even,}\\
    C_{p,k}|x|^{p-k} & \text{for}& k \text{ odd,}
\end{array}\right.
\]
thus \[f(v_n+z_n)-f(v_n)=\sum_{k=1}^{\alpha_p} f^{(k)}(v_n)z_n^k + \underbrace{\int_{v_n}^{v_n+z_n} \frac{f^{(\alpha _p+1)}(s)}{\alpha _p!} (v_n+z_n-s)^{\alpha _p} \, \mathrm{d}s}_{R(z_n,v_n)}\,.\] 
One can integrate \eqref{1eqDerivee} and use $E(0)=0$ and $(v_n(0),\partial_t v_n(0))=(0,0)$ so that we can write: 
\begin{equation}\label{1eqIntegree}
   E_n(t) = \sum_{k=1}^{\alpha_p} C_{k,p}I_n^{(k)} + C_pR_n \,, 
\end{equation}
where 
\begin{gather*}
    I_n^{(k)}:= -\int_0^t \int_{\mathbb{R}^3} \partial _t v_n  f^{(k)}(v_n)z_n^k \, \mathrm{d}x \, \mathrm{d}t' \text{  for  } 1 \leqslant k \leqslant \alpha_p,\\
    R_n:= -\int_0^t \int_{\mathbb{R}^3}\partial_t v_n R(z_n,v_n) \, \mathrm{d}x \, \mathrm{d}t'\,.
\end{gather*}
We first estimate $R_n$ as we expect it to be simpler to handle. Remark that for $\theta_n \in [v_n,v_n+z_n]$ one has \[|f^{(\alpha _p+1)}(\theta_n)| \lesssim |v_n+z_n|^{p- \alpha_p -1} + |v_n|^{p- \alpha _p -1} \lesssim |v_n|^{p- \alpha_p -1} + |z_n|^{p- \alpha _p -1} \,,\] so that \[R(v_n,z_n) \lesssim |z_n|^p + |v_n|^{p-\alpha_p -1}|z_n|^{\alpha _p +1}\,.\]
The Hölder inequality and the Young inequality give: 
\begin{align*}
    R_n & \lesssim \int_0^t \|\partial _t v_n(t')\|_{L^2} \|v_n(t')\|^{p-\alpha_p-1}_{L^{p+1}} \|z_n(t')\|^{\alpha _p+1}_{L^{r_p(\alpha _p +1)}}\, \mathrm{d}t' + \int_0^t \|\partial _t v_n(t')\|_{L^2} \|z_n(t')\|_{L^{2p}}^{p}\, \mathrm{d}t' \\
    & \lesssim \int_0^t \|\partial _t v_n(t')\|_{L^2} \|v_n(t')\|^{p-\alpha_p-1}_{L^{p+1}} \|z_n(t')\|^{\alpha_p+1}_{L^{r_p(\alpha _p +1)}}\, \mathrm{d}t' + \int_0^t \|\partial _t v_n(t')\|^2_{L^2} \, \mathrm{d}t' + \|z_n\|_{L^{2p}_TL^{2p}}^{2p} \\ 
    & \lesssim  \|z_n\|_{L^{2p}_TL^{2p}}^{2p} + \left(1+ \|z_n\|_{L^{\infty}_T L^{r_p(\alpha_p +1)}}^{\alpha_p +1}\right) \int_0^t \max \left\{E_n(t'),E_n(t')^{\frac{1}{2}+\frac{p-\alpha_p-1}{p+1}}\right\}\, \mathrm{d}t'\,.
\end{align*}
Observe that $\frac{p-\alpha_p -1}{p+1} \leqslant \frac{1}{2}$ since $\alpha _p = \lceil \frac{p-3}{2} \rceil \geqslant \frac{p-3}{2}$ so that 
\[R_n \lesssim \|z_n\|_{L^{2p}_TL^{2p}}^{2p}  + \left(1+ \|z_n\|_{L^{\infty}_T L^{r_p(\alpha_p+1)}}^{\alpha_p +1}\right) \int_0^t (1+E_n(t')) \, \mathrm{d}t'\,.\]

Bounds for the terms $I_n^{(k)}$ require a more intricate analysis and will follow Lemma~\ref{1lemmeTechnique}. More precisely, let $1 \leqslant k \leqslant \alpha_p$. First apply Fubini's theorem and write: \[I_n^{(k)}=- \int_{\mathbb{R}^3} \int_0^t \partial _t (f^{(k-1)}(v_n)) z_n^k \, \mathrm{d}t' \, \mathrm{d}x.\] 

Integrate by parts in time so that
\[I_n^{(k)} = -\int_{\mathbb{R}^3} f^{(k-1)}(v_n(t))z_n^k(t)\, \mathrm{d}x + k\int_{\mathbb{R}^3} \int_0^t f^{(k-1)}(v_n(t')) \partial_t z_n(t') z_n(t')^{k-1} \, \mathrm{d}t' \, \mathrm{d}x.\] 
Observe that $\partial _t z_n=\langle \nabla \rangle \tilde{z}_n$ and bound
\begin{align*}
|I_n^{(k)}| &\lesssim \int_{\mathbb{R}^3} |z_n(t)|^{k}|v_n(t)|^{p-k+1} \, \mathrm{d}x   +  \left \vert \int_0^t \int_{\mathbb{R}^3} f^{(k-1)}(v_n(t')) \langle \nabla \rangle \tilde{z}_n(t') z_n(t')^{k-1}\, \mathrm{d}t' \, \mathrm{d}x \right \vert \\
&= J_n^{(k)}+K_n^{(k)}.
\end{align*}
In order to handle $J_n^{(k)}$, we use Hölder and Young's inequality:
\[J_n^{(k)}(t) \leqslant  E_n(t)^{\frac{p-k+1}{p+1}}\|z_n\|^k_{L^{p+1}} \leqslant \frac{1}{2}E_n(t) + C\|z_n\|_{L^{\infty}_TL^{p+1}}^{p+1}.\]
$K_n^{(k)}$ is more difficult to study and is estimated \textit{via} Lemma~\ref{1lemmeTechnique}, so that \[K_n^{(k)} \lesssim g\left(\|z_n\|_{L^{\infty}((-T,T),X)},\|\tilde z_n\|_{L^{\infty}((-T,T),Y)}\right) \left( 1+ \int_0^t E_n(t')\,\mathrm{d}t'\right).\]
Finally using the bounds from \eqref{1eqProbaSet} we have \begin{align*}
E_n(t) &\lesssim  \left(1+ \|z_n\|_{L^{\infty}_T L^{r_p(\alpha_p+1)}}^{\alpha_p +1} + g\left(\|z_n\|_{L^{\infty}((0,T),X)},\|\tilde z_n\|_{L^{\infty}((0,T),Y)}\right)\right) \left(1+\int_0^t E_n(t')\,\mathrm{d}t' \right) \\
&+ \|z_n\|_{L^{2p}_TL^{2p}}^{2p} + \|z_n\|_{L^{\infty}_TL^{p+1}}^{p+1} \\
& \lesssim 1+ \int_0^t E_n(t')\, \mathrm{d}t'\,.
\end{align*}
Knowing that the implicit constant does not depend on $n$, but only on $\eta, T, p$, the Grönwall lemma ends the proof. 
\end{proof}

It remains to prove Lemma~\ref{1lemmeTechnique}. Its proof will require a chain rule estimate in Besov spaces whose proof is similar to the one of Theorem~2.61 in \cite{bahouriCheminDanchin}. 

\begin{lemma}[Chain rule estimates in Besov spaces]\label{1chainRule} Let $u \in \mathcal{S}'$, $s \in (0,1)$ and $p>3$. Let $q,r \in[1,\infty]$ and $q_1,q_2 \in [1,\infty]$ satisfying $\frac{1}{q}=\frac{1}{q_1}+\frac{1}{q_2}$. Let $f$ denote the function defined by $f(x)=x|x|^{p-1}$ or $f(x)=\operatorname{sgn}(x)x|x|^{p-1}=|x|^p$. Then the following identities hold: 
\begin{enumerate}[label=(\textit{\roman*})]
    \item $\|f(u)\|_{B^s_{q,r}} \lesssim \|u\|_{B^s_{q_1,r}} \||u|^{p-1}\|_{L^{q_2}}$, and $\|f(u)\|_{\dot B^s_{q,r}} \lesssim \|u\|_{\dot B^s_{q_1,r}} \||u|^{p-1}\|_{L^{q_2}}$.
    \item $\|f(u)\|_{B^s_{q,r}} \lesssim_{\varepsilon} \|u\|_{W^{s+\varepsilon,q_1}} \||u|^{p-1}\|_{L^{q_2}}$ for all $\varepsilon>0$ such that $s+\varepsilon <1$.
\end{enumerate}
\end{lemma}

\begin{proof}[Proof of Lemma~\ref{1chainRule}] The proof uses Lemma~\ref{1lemmaReconstruction}. Since $\mathbf{P}_ju \underset{j \to \infty}{\longrightarrow} u$ in $L^p$ and $f(0)=0$ we have \[f(u)= \sum_{j \geqslant 0} f_j \text{ with } f_j:= f(\mathbf{P}_{j+1}u) - f(\mathbf{P}_j u).\] The Taylor formula at order $1$ writes: \[f_j=\Delta_ju \; m_j \text{ with } m_j:=\int_0^1 f'(\mathbf{P}_ju +t \Delta_j u)\,\mathrm{d}t.\] In view of Lemma~\ref{1lemmaReconstruction} we will focus on estimating $\|\partial ^{\alpha} f_j\|_{L^p}$ with $|\alpha| \leqslant \lfloor s \rfloor +1 =1$. For $|\alpha |=0$ write that $\|\partial ^{\alpha} f_j\|_{L^p} \leqslant \|m_j\|_{L^{q_2}} \|\Delta_j u\|_{L^{q_1}}$. Then we estimate $m_j$:
\[|m_j|  \leqslant p \int_0^1 |\mathbf{P}_ju+t\Delta _ju|^{p-1}\, \mathrm{d}t \lesssim |\mathbf{P}_ju|^{p-1} + |\Delta _j u|^{p-1}\]
where we used that $|f'(x)| \leqslant p|x|^{p-1}$ for $x \in \mathbb{R}$ and $(a+b)^{p-1} \lesssim a^{p-1} + b^{p-1}$ for $a,b >0$. 

Then $\|m_j\|_{L^{q_2}} \lesssim \|\mathbf{P}_ju\|_{L^{q_2(p-1)}}^{p-1}$. 
Similarly, when $|\alpha|=1$ and for multi-indicies $\beta \leqslant \alpha$, the Bernstein inequality (the function to which it is applied is indeed with frequencies supported in a ball of radius $\simeq 2^j$) and the same arguments as before yield 
\[\|\partial^{\beta} m_j\|_{L^{q_2}} \lesssim 2^{j|\beta|} \left \| \int_0^1 f'(\mathbf{P}_ju + t \Delta_j u)\, \mathrm{d}t \right \|_{L^{q_2}} \lesssim 2^{j |\beta|} \|\mathbf{P}_ju\|^{p-1}_{L^{q_2(p-1)}}\,.\] Using the Leibniz formula $\partial ^{\alpha} (fg) = \sum_{\beta \leqslant \alpha} \binom{\alpha}{\beta} \partial ^{\alpha - \beta} f \partial ^{\beta} g$ and putting all the previous estimates together we recover the estimate \[\sup_{\alpha \leqslant \lfloor s \rfloor +1} 2^{j(s-|\alpha|)}\|\partial ^{\alpha} u_j\|_{L^q} \lesssim  2^{js}\|\Delta _j u\|_{L^{q_1}} \|\mathbf{P}_ju\|^{p-1}_{L^{q_2(p-1)}}\,.\]
The proof of (\textit{i}) now follows from the direct inequality $\|\mathbf{P}_ju\|_{L^{q_2(p-1)}} \lesssim \|u\|_{L^{q_2(p-1)}}$.

For the homogeneous counterpart of (\textit{i}), replace all the appearences of $\mathbf{P}_j$ or $\Delta _j$ with $\dot{\mathbf{P}}_j$ or~$\dot\Delta _j$ and observe that all the inequalities written still hold true. The only difficulty is proving the convergence of the series $\sum_{j \leqslant 0} f_j$ with $f_j:=f(\dot{\mathbf{P}}_{j+1}u)-f(\dot{\mathbf{P}}_{j}u)$. This is explained in \cite{bahouriCheminDanchin}, Lemma 2.62. 

For (\textit{ii}) write $2^{js}=2^{j(s+\varepsilon)}2^{-js\varepsilon}$, and use the Hölder inequality to obtain: 
\begin{equation}
    \label{1eqB}
    \|f(u)\|_{B^{s}_{q,r}} \lesssim \|u\|_{B^{s+\varepsilon}_{q_1,\infty}} \sum_{j \geqslant 0} 2^{-j\varepsilon} \|\mathbf{P}_ju\|_{L^{q_2(p-1)}}^{p-1}.
\end{equation}
Then Theorem~\ref{1sobolev} gives \[\|u\|_{B^{s+\varepsilon}_{q_1,\infty}} \lesssim \|u\|_{W^{s+\varepsilon,q_1}} \text{ and } \|\mathbf{P}_ju\|_{L^{q_2(p-1)}}^{p-1} \lesssim_{\varepsilon} \|u\|_{L^{q_2(p-1)}}^{p-1}\,.\] These inequalities and \eqref{1eqB} end the proof. Note that there is no homogeneous counterpart of~(\textit{ii}).
\end{proof}

\begin{proof}[Proof of Lemma~\ref{1lemmeTechnique}] We now turn the idea explained in the introduction into a mathematical proof. An efficient way of doing so is the systematic use of the Littlewood-Paley theory. Recall that $s_p=\frac{p-3}{p-1}$. In the following, estimates for $k=1$ and $k \geqslant 2$ could be different, thus we assume $k \geqslant 2$ and explain the modifications for $k=1$ at the end. The Fourier-Plancherel theorem and the fact that the contribution for $j'=-1,0,1$ are up to a universal constant identical to the case $j'=0$ yield:
\begin{align*}
    \lefteqn{\left \vert \int_{\mathbb{R}^3} f^{(k-1)}(v_n(t)) z_n(t)^{k-1} \langle \nabla \rangle \tilde{z}_n(t) \, \mathrm{d}x \right \vert} \qquad \qquad \qquad &\\& = \left \vert \sum_{j'=-1}^{1} \sum_{j \geqslant 0} \int_{\mathbb{R}^3} \Delta_{j}(f^{(k-1)}(v_n(t)) z_n(t)^{k-1})\Delta_{j+j'}(\langle \nabla \rangle \tilde{z}_n(t)) \, \mathrm{d}x \right \vert \\
    & \lesssim \sum_{j > 2} \int_{\mathbb{R}^3} |\Delta_{j}(f^{(k-1)}(v_n(t)) z_n(t)^{k-1})||\Delta_j(\langle \nabla \rangle \tilde{z}_n(t) )| \, \mathrm{d}x \\
    & + \sum_{j=0}^2 \int_{\mathbb{R}^3}|\Delta_{j}(f^{(k-1)}(v_n(t)) z_n(t)^{k-1}(t))||\Delta_j(\langle \nabla \rangle \tilde{z}_n(t) )| \, \mathrm{d}x \\
    & =: I_1 + I_2\,,
\end{align*}
We first estimate $I_2$ using Hölder, Bernstein and Young's inequalities, where $r_k:= \frac{(k-1)(p+1)}{k}$:
\begin{align*}
    I_2 &\lesssim \|z_n(t)\|^{k-1}_{L^{r_k}} \|v_n(t)\|_{L^{p+1}}^{p-k+1} \sum_{j=0}^2 \|\langle \nabla \rangle  \Delta _j \tilde z_n(t)\|_{L^{\infty}_x} \\
    & \lesssim \|z_n(t)\|^{k-1}_{L^{r_k}} \|\tilde z_n(t)\|_{L^{\infty}_x} E_n(t)^{\frac{p-k+1}{p+1}}\\
    & \lesssim E_n(t) + \|z_n\|^{r_k}_{L^{\infty}_TL_x^{r_k}} \|\tilde z_n\|_{L^{\infty}_TL^{\infty}_x}^{\frac{p+1}{k}}\,. 
\end{align*}
For $I_1$ observe that with Hölder and Bernstein inequalities, 
\[I_1 \lesssim \sum_{j > 2} 2^{j(1-s_p)}\|\Delta_j(z_n(t)^{k-1} f^{(k-1)}(v_n(t)))\|_{L^{1}}2^{js_p}\|\Delta_j( \tilde z_n(t))\|_{L^{\infty}}\] 
so that the Hölder inequality for series gives, as only the high frequencies appeared in the sum, \[I_1 \lesssim \|z_n(t)^{k-1}f^{(k-1)}(v_n(t))\|_{\dot B^{1-s_p}_{1,\infty}}\|\tilde z_n(t)\|_{B^{s_p}_{\infty,1}}\,.\] 
Then Corollary~\ref{1coroProductRule} provides us with:
\begin{align*}
 \|z_n(t)^{k-1}&f^{(k-1)}(v_n(t))\|_{\dot B^{1-s_p}_{1,\infty}} \lesssim \|f^{(k-1)}(v_n(t))\|_{\dot B^{1-s_p}_{\frac{p+1}{p+2-k},\infty}}\|z_n(t)^{k-1}\|_{L^{\frac{p+1}{k-1}}} \\ 
 &+ \||v_n(t)|^{p-k+1}\|_{L^{\frac{p+1}{p-k+1}}}\|z_n(t)^{k-1}\|_{\dot B^{1-s_p}_{q_k,\infty}}\\
 &\lesssim \underbrace{\|f^{(k-1)}(v_n(t))\|_{\dot B^{1-s_p}_{\frac{p+1}{p+2-k},\infty}}}_{J_1}\|z_n(t)\|^{k-1}_{L^{p+1}} + E_n(t)^{\frac{p-k+1}{p+1}}\underbrace{\|z_n(t)^{k-1}\|_{\dot B^{1-s_p}_{q_k,\infty}}}_{J_2}\,, 
\end{align*}
where we recall that $\frac{1}{q_k}+\frac{p-k+1}{p+1}=1$. The chain rule from Lemma~\ref{1chainRule} will estimate the terms~$J_1$ and~$J_2$. For $J_2$, a direct application shows that: \[J_2 \lesssim \|z_n(t)\|_{\dot{B}^{1-s_p}_{q_k,\infty}} \|z_n\|_{L^{\infty}}^{k-2}.\] For $J_1$, also using Lemma~\ref{1chainRule} followed by Theorem~\ref{1sobolev}~(\textit{ii}) and then by the Gagliardo-Nirenberg inequality, Theorem~\ref{1gagliardoNirenberg} yield:
\begin{align*}
    J_1 & \lesssim  \|v_n(t)\|_{\dot B^{1-s_p}_{\frac{p+1}{2},\infty}}\||v_n(t)|^{p-k}\|_{L^{\frac{p+1}{p-k}}} \\
    & \lesssim \|v_n(t)\|_{\dot W^{1-s_p, \frac{p+1}{2}}}E_n(t)^{\frac{p-k}{p+1}} \\
    & \lesssim \|\nabla v_n(t)\|_{L^2}^{1-\alpha} \|v_n(t)\|_{L^{p+1}}^{\alpha} E_n(t)^{\frac{p-k}{p+1}},
\end{align*}
where $\alpha \in [0,s_p]$ is such that $\frac{2}{p+1}= \frac{1-s_p}{3} +\frac{1-\alpha}{6}+\frac{\alpha}{p+1}$, \textit{i.e} $\alpha = s_p = \frac{p-3}{p-1}$. Finally we get: 
\begin{equation}
\label{1major}
J_1 \lesssim E_n(t)^{\frac{p-k}{p+1} + \frac{\alpha}{p+1}+\frac{1-\alpha}{2}} = E_n(t)^{\frac{p-k}{p+1} + \frac{2}{p+1}}.
\end{equation}
This yields $J_1 \lesssim 1+E_n(t)$. 

For $k=1$ one can proceed in the same manner: split the left-hand side of \eqref{1estimeeTechnique} into $I_1$ and $I_2$ and write $I_2 \lesssim E_n(t)+\|\tilde{z}_n\|_{L^{\infty}_TL^{\infty}_x}^{p+1}$. Remark that the estimate of $I_1$ is handled using the same arguments as for $J_1$ above leading to \[I_1 \lesssim \|f(v_n(t))\|_{\dot B^{1-s_p}_{1,\infty}} \lesssim E_n(t),\] which ends the proof. 
\end{proof} 

\subsection{Passing to the limit and end of the proof}\label{1sousSectionLimit}

The linear part is handled \textit{via} the following elementary lemma:

\begin{lemma}[Linear compactness]\label{1linearLemma} For every $q \in [1,\infty]$ and $r \in (1,\infty)$ one has \[\|z_{n} - z\|_{L^{q}((-T,T),L^r(\mathbb{R}^3))} \underset{n \to \infty}{\longrightarrow}0\,.\] 
\end{lemma}

\begin{proof} By definition we have $z_n-z=\left(\operatorname{id}-\mathbf{P}_n\right)z$. As $\mathbf{P}_n$ is a mollifier it follows that for every $t \in (-T,T)$, $\|z_n(t)-z(t)\|_{L^r} \to 0$. Now observe that $\|z_n(t)-z(t)\|_{L^r} \leqslant 2 \|z(t)\|_{L^r}$ and $z \in L^{q}((-T,T),L^r(\mathbb{R}^3))$ so that the Lebesgue convergence theorem gives the desired result. 
\end{proof}

The nonlinear part $v_n$ will be handled using the following compactness result. 

\begin{lemma}[Nonlinear compactness]\label{1nonLinear} There exists a function $v$ that belongs to the space: \[H^1((-T,T) \times \mathbb{R}^3) \cap L^{p+1}((-T,T) \times \mathbb{R}^3) \cap \mathcal{C}^0([-T,T],L^2(\mathbb{R}^3))\] such that up to extraction: 
\begin{enumerate}[label=(\textit{\roman*})]
    \item $v_n \underset{n \to \infty}{\longrightharpoonup} v$ in $H^1((-T,T) \times \mathbb{R}^3)$,
    \item $v_n \underset{n \to \infty}{\longrightarrow} v$ in $L^2_{\text{loc}}((-T,T)\times \mathbb{R}^3)$,
    \item $v_n \underset{n \to \infty}{\longrightarrow} v$ in $L^{p}_{\text{loc}}((-T,T) \times \mathbb{R}^3)$.
\end{enumerate}
\end{lemma}

\begin{proof} (\textit{i}) follows from the boundedness of $(v_n)_{n \geqslant 0}$ in the space $H^1((0,T) \times \mathbb{R}^3)$ and the Banach-Alaoglu theorem in Hilbert spaces. This bound is indeed obtained via the energy control of $v$ which immediately implies \[\sup_{n \geqslant 0} \left\{ \|\nabla v_n\|_{L^{\infty}((-T,T),L^2)} + \|\partial _t v_n\|_{L^{\infty}((-T,T),L^2)} \right\} < \infty.\] Since $(-T,T)$ is a bounded interval, we obtain $\sup_{n \geqslant 0} \|v_n\|_{\dot H ^1((-T,T) \times \mathbb{R}^3)} < \infty$. The Taylor formula in time gives \[\|v_n(t)\|_{L^2} \leqslant \|v_n(0)\|_{L^2} + \int_{-T}^T\|\partial _t v_n(t')\|_{L^2} \, \mathrm{d}t'.\] Using the $L^{\infty}((-T,T),L^2)$ bound for $\partial_t v_n$ yields a uniform bound for $v_n$ in $L^{\infty}((-T,T),L^2)$ and thus in $L^{2}((-T,T),L^2)$. 

Finaly the sequence $(v_n)_{n \geqslant 0}$ is bounded in the space $H^1((-T,T) \times \mathbb{R}^3)$. The Banach-Alaoglu theorem proves that up to extraction we can assume that $v_n$ is weakly convergent to a function $v$ that belongs to $H^1((-T,T) \times \mathbb{R}^3)$. Note that the uniform bound for $\partial_t u_n$ in $L^{\infty}((0,T),L^2)$ allow to use the Ascoli theorem which ensures that $v \in \mathcal{C}^0([0,T],L^2)$ 

(\textit{ii}) Let $K \subset (T,T) \times \mathbb{R}^3$ be a compact set. Then the fact that the embedding $H^1 \hookrightarrow L^2_{\text{loc}}$ is compact (this is the Rellich-Kondrakov theorem), and the bound from (\textit{i}) proves that up to another extraction, (\textit{ii}) holds. Up to a diagonal extraction we can assume that this sequence converges for any compact set $K$. 

(\textit{iii}) We have proved local compactness for $(v_n)_{n \geqslant 0}$ in $L^2$ in both space and time, and a uniform bound for $(v_n)_{n \geqslant 0}$ in $L^{p+1}$ given by Proposition~\ref{1propProbabilistic}. We can interpolate those two, and for every compact $K$: \[\|v_n-v\|_{L^{p}(K)} \lesssim \|v_n-v\|^{\alpha}_{L^2(K)}\|v_n-v\|^{1-\alpha}_{L^{p+1}(K)} \lesssim \|v_n-v\|^{\alpha}_{L^2(K)},\] with $\alpha \in (0,1)$ such that $\frac{3}{p}= \frac{\alpha}{2}+\frac{1-\alpha}{p+1}$. This proves the convergence. 
\end{proof}

We are now ready for the proof of Theorem~\ref{1mainTheorem}.
\begin{proof}[End of the proof of Theorem~\ref{1mainTheorem}] Without loss of generality, assume that $s \in (\frac{p-3}{p-1},1)$, as for $s \geqslant 1$, $\mathcal{H}^s$ initial data are also in $\mathcal{H}^{1-\varepsilon}$. 

Set $\Omega_{T, \frac{\eta}{2}}$ as in Proposition~\ref{1propProbabilistic} and consider $\Omega _{T,\frac{\eta}{2}}' := \{\|\langle \nabla \rangle^{s} z\|_{L^{p+1}((-T,T) \times \mathbb{R}^3)} \leqslant \lambda\}$ with $\lambda >0$ large enough to ensure that $\mathbb{P}(\Omega _{T,\frac{\eta}{2}}') \geqslant \frac{\eta}{2}$. Now set $\tilde{\Omega}_{T,\eta } := \Omega_{T, \eta /2} \cap \Omega_{T, \eta /2}'$ so that $\mathbb{P}(\tilde{\Omega}_{T,\eta}) \geqslant 1- \eta$. Now we will only deal with intial data randomization arising from $\tilde{\Omega}_{T,\eta}$. This in particular enables to use the compactness lemmata proven before. Take $\varphi$ an admissible test function from Defintion~\ref{1weakSolution}. The weak convergence $u_n \rightharpoonup u$ in $H^1((-T,T) \times \mathbb{R}^3)$, the strong convergence $z_n \rightarrow z$ and $v_n \rightarrow v$ both in $L^p_{\text{loc}}((-T,T) \times \mathbb{R}^3)$ and the fact that $\varphi$ is compactly supported in space and time in $(-T,T)$ proves that 
\begin{align}
    0& =\int_{-T}^T \int_{\mathbb{R}^3} \left( \partial _t v_n(t) \partial _t \varphi (t) - \nabla v_n(t) \cdot \nabla \varphi (t) - (z_n(t)+v_n(t))|z_n(t)+v_n(t)|^{p-1} \varphi (t) \right) \, \mathrm{d}x \, \mathrm{d}t \notag \\
    & \underset{n \to \infty}{\longrightarrow}\int_{-T}^T \int_{\mathbb{R}^3} \left( \partial _t v(t) \partial _t \varphi (t) - \nabla v(t) \cdot \nabla \varphi (t) - (z(t)+v(t))|z(t)+v(t)|^{p-1} \varphi (t) \right) \, \mathrm{d}x \, \mathrm{d}t =0. \label{1eqT}
\end{align}

We have proved that for each $\eta >0$ there exists a set $\tilde{\Omega}_{T,\eta}$ with measure greater than $1- \eta$ such that for initial random data generated with $\omega \in \tilde{\Omega}_{T,\eta}$ there exists a weak solution to \eqref{1SLW} on the time interval $(-T,T)$. Now apply the above with $\eta :=\frac{1}{n^2}$ for each $n \geqslant 2$ and set $A_n := \Omega_{n,\frac{1}{n^2}}^c$. Then $\sum _{n \geqslant 0} \mathbb{P}(A_n) < + \infty$ and by the Borel-Cantelli lemma it follows that $\mathbb{P}(\limsup A_n)=0$, where $\limsup A_n := \bigcap_{n \geqslant 0} \bigcup_{k \geqslant n} A_k$ so that $\Omega_{T}:=(\limsup A_n)^c$ is a set of probability $1$ where existence of weak solutions on $(-T,T)$ is granted. Finaly we set $\tilde{\Omega} := \bigcap_{n \geqslant 1} \Omega_n$ which is of probability one on which a global weak solution exists and satisfies \eqref{1eqT} for every $T>0$ and every compactly supported test function $\varphi \in \mathcal{C}^2((-T,T) \times \mathbb{R}^3)$. 
\end{proof}

The proof of the continuity in time part of Corollary~\ref{1coroAdditional} is a consequence of the Aubin-Lions compactness theorem that we recall. For a proof see \cite{boyer}. 

\begin{theorem}[Aubin-Lions]\label{1aubinLions} Let $X_0 \hookrightarrow X \hookrightarrow X_1$ be three Banach spaces, the first embedding being compact and the second being continuous. Let $p,q \in [1,\infty]$ and \[W=\{u \in L^p((-T,T),X_0), \partial _t u \in L^q((-T,T),X_1)\}\,.\] Then:
\begin{enumerate}[label=(\textit{\roman*})]
    \item If $p<\infty$ then $W \hookrightarrow L^p((-T,T),X)$ is compact. 
    \item If $p=\infty$ and $q>1$ then $W \hookrightarrow \mathcal{C}^0([-T,T],X)$ is compact. 
\end{enumerate}
\end{theorem}

\begin{proof}[Proof of Corollary~\ref{1coroAdditional}] Let us first prove the continuity in time. As the linear solution $(z,\partial_tz)$ has regularity $\mathcal{C}^0(\mathbb{R},\mathcal{H}^s)$, thanks to Lemma~\ref{1linearLemma}, it is sufficient to prove the needed continuity on $(v,\partial _t v)$ on every interval $[-T,T]$. Fix such an interval and let $z_n,v_n$ be the regularized solutions introduced in~\eqref{1rSLW}. Recall that they satisfy 
\begin{equation}
    \label{1regEq}
    \partial _t^2v_n = \mathbf{P}_n \Delta v_n - \mathbf{P}_n((z_n+v_n)|z_n+v_n|^{p-1}) 
\end{equation}
We will prove the two continuity results:
\begin{enumerate}[label=(\textit{\roman*})]
    \item $v \in \mathcal{C}^0([-T,T],H^s)$,
    \item $\partial_t v \in \mathcal{C}^0([-T,T],H^{s-1})$. 
\end{enumerate}

(\textit{i}) results from the uniform bound for $(v_n)_{n \geqslant 1}$ in $L^{\infty}((-T,T),H^1)$, the uniform bound for $(\partial _t v_n)_{n \geqslant 1}$ in $L^{\infty}((-T,T),L^2)$ given by Proposition~\ref{1propProbabilistic}, and the Aubin-Lions Theorem~\ref{1aubinLions} with $X=H^s$, $X_0=H^1$, $X_1=L^2$, $p=q=\infty$. 

(\textit{ii}) We use that \[\sup_{n \geqslant 1} \|\partial _t v_n\|_{L^{\infty}((-T,T),L^2)} < \infty\,.\] We will also need the estimate 
\begin{equation}
    \label{1order2}
    \sup_{n \geqslant 1} \|\partial^2 _t v_n\|_{L^{\infty}((-T,T),H^{-3/2})} < \infty
\end{equation}
so that another application of the Aubin-Lions theorem proves the needed continuity. In order to prove~\eqref{1order2} remark that as $H^{-1} \hookrightarrow H^{-3/2}$ we have \[\|\mathbf{P}_n\Delta v_n(t)\|_{H^{-3/2}} \lesssim \|\Delta v_n(t)\|_{H^{-1}} \lesssim \|v_n(t)\|_{H^1}\] so that
\begin{equation}
    \label{1order21}
    \sup_{n \geqslant 1} \|\mathbf{P}_n\Delta v_n\|_{L^{\infty}((-T,T),H^{-3/2})} \lesssim \sup_{n \geqslant 1} \|v_n\|_{L^{\infty}((-T,T),H^1)} < \infty\,.
\end{equation}
Remark that thanks to Proposition~\ref{1propProbabilistic}, $\left((z_n+v_n)|z_n+v_n|^{p-1}\right)_{n \geqslant 1}$ is uniformly bounded in $L^{\frac{p+1}{p}} \hookrightarrow H^{-\frac{3(p-1)}{2(p+1)}} \hookrightarrow H^{-3/2}$. Combined with \eqref{1order21} and \eqref{1regEq} gives the uniform bound \eqref{1order2} and ends the proof.

For the first part of Corollary~\ref{1coroAdditional} we need to find an invariant set of full $\mu$-measure. Consider the set \[\Theta := \{(u_0,u_1) \in \mathcal{H}^s, \|S(t)(u_0,u_1)\|_{X\cap Y \cap L^2 \cap L^{\infty}} \in L^{\infty}_{\text{loc}}(\mathbb{R})\}\,.\] 
Set $\Sigma := \Theta + \mathcal{H}^1(\mathbb{R}^3)$. This is indeed a set invariant by the flow as $S(t)(\Theta)=\Theta$. Moreover this set is of full $\mu$-measure since $\Theta$ is of full measure by the proof of Theorem~\ref{1mainTheorem}. This set also gives rise to weak solutions.  
\end{proof}

Finally we prove the finite speed of propagation when $U=\mathbb{R}^3$. 

\begin{proof}[Proof of Corollary~\ref{1finiteSpeed}] For a given $s > \frac{p-3}{p-1}$, let $(u_0,u_1) \in \mathcal{H}^s$ with compact support included in $B(0,R)$, an initial data that gives rise to a global solution constructed in the proof of Theorem~\ref{1mainTheorem}. We want to prove that the solution $u(t)$ to \eqref{1SLW} is also compactly supported, in $B(0,R+t)$. 

The finite speed of propagation is known to hold for solutions to \eqref{1SLW} as well as solutions to \eqref{1rSLW} as soon as the initial data belongs to the energy space $\mathcal{H}^1$, and propagation holds with maximum speed $1$. This proves that the approximate solutions $u_n(t)$ are supported in $B\left(0,R+t\right)$. As we know that $v_n(t) \to v$ almost everywhere, and that up to extraction $z_n(t) \to z(t)$ almost everywhere, $u(t)=z(t)+v(t)$ is an almost everywhere pointwise limit of the $u_n(t)$, consequently $\operatorname{supp} u(t) \subset B(0,R+t)$. 
\end{proof}

\section{Proof of Theorem~\ref{1th2}}\label{1sectionTh2} The proof of Theorem~\ref{1th2} uses Proposition~\ref{1propProbabilistic5} which proof, as explained in the introduction, differs from the one of Proposition~\ref{1propProbabilistic} in avoiding any $L^{\infty}_T$ and $L^{\infty}_x$ estimates on the linear part $z_n,\tilde{z}_n$.  

Let $(u_0,u_1) \in \mathcal{H}^{s_p}$ and $z_n$ be the solution to the linear equation \eqref{1LW} with initial data $(z_n(0),\partial _t z_n(0))=(\mathbf{P}_n u_0, \mathbf{P}_n u_1)$ and $v_n$ the unique smooth global solution to the perturbed nonlinear wave equation, which is energy subcritical ($p<5$): 

\begin{equation}
\label{1r}
    \left \{
    \begin{array}{cc}
         \partial_t^2 v_n - \Delta v_n + |z_n+v_n|^{p-1}(z_n+v_n)=0, \\
         (v_n(0), \partial _t v_n(0))=(0,0).
    \end{array}
    \right.
\end{equation} 

We recall, as above, the energy of the nonlinear part $v_n$:  \[E_n(t) := \int_{\mathbb{R}^3} \left( \frac{|\partial _t v_{n}(t)|^2}{2} + \frac{|\nabla v_{n}(t)|^2}{2} + \frac{|v_n(t)|^{p+1}}{p+1}\right) \, \mathrm{d}x\,.\]

We state the main result of this section. 

\begin{proposition}[Probabilistic \textit{a priori} estimates for $s=s_p$]\label{1propProbabilistic5} Let $p \in (3,5)$, $T>0$ and $\eta \in (0,1)$. There exists a measurable set $\Omega_{T,\eta} \subset \Omega$ and a constant $C(T, \eta, \|(u_0,u_1)\|_{\mathcal{H}^{s_p}})$ which depends only on $T,\eta, \|(u_0,u_1)\|_{\mathcal{H}^{s_p}}$ such that:
\begin{enumerate}[label=(\textit{\roman*})]
    \item $\mathbb{P}(\Omega_{T,\eta}) \geqslant 1- \eta$.
    \item For every $\omega \in \Omega_{T,\eta}$, if the initial data for $u_n$ is attached to $\omega$ \textit{via} the randomization map of $(u_0,u_1)$ then: \[\sup_{n \geqslant 0} \sup_{t \in(0,T)} E_n(t) \leqslant C(T, \eta, \|(u_0,u_1)\|_{\mathcal{H}^{s_p}}).\]
\end{enumerate}
\end{proposition}

Once Proposition~\ref{1propProbabilistic5} is obtained, the proof of Theorem~\ref{1th2} follows from the deterministic theory developed in~\cite{sunXia} which asserts that local solutions exists in $\mathcal{H}^s$ and can be globalized under energy control conditions. We refer to~\cite{sunXia} for details. 

\begin{proposition}[Reduction to global bounds] Let $v_n$ be the solution to~\ref{1r}. Then if ther exists $C(T,\|-u_0,u_1)\|_{\mathcal{H}^s})< \infty$ such that 
\[\sup_{n \geqslant 0} \sup_{t \in (0,T)} E(t) \leqslant C(T, \|(u_0,u_1)\|_{\mathcal{H}^s})\,,\]
then ther exists a unique global solution $u := z + v$ to~\ref{1SLW}, where we recall that $z$ solve~\ref{1LW}. 
\end{proposition}

\begin{remark} This reduction to global bounds is typical of the subcritical equations. Here $p<5$ is energy subcritical, thus \eqref{1r} enjoys a subcritical well-posedness theory and maximal time depend only on the size of the norm. In the critical case, $p=5$ a similar results holds, which is much harder to prove, see~\cite{pocovnicu}, Section~4. 
\end{remark}

It remains to prove Proposition~\ref{1propProbabilistic5}. 

\begin{proof}[Proof of Proposition~\ref{1propProbabilistic5}] We first establish energy estimates and will fix the probabilistic setting later. 

The proof begins with the same energy estimates treated in Proposition~\ref{1propProbabilistic}, using the same integration in time technique and a Taylor expansion of $|x|^{p-1}x$ at order $1$. We omit the details and obtain: 
\begin{align*}
   E_n(t)& = \int_0^t\int_{\mathbb{R}^3} \langle \nabla \rangle \tilde{z}_n(t') v_n^{p}(t') \, \mathrm{d}x \, \mathrm{d}t' - \int_{\mathbb{R}^3} z_n(t')v_n^p(t') \, \mathrm{d}x \\ 
   &- \int_0^t \int_{\mathbb{R}^3} \partial _t v_n(t')N(z_n,v_n)(t') \, \mathrm{d}x \, \mathrm{d}t' \\
&=:J_1+J_2+J_3 
\end{align*}
where $|N(z_n,v_n)(t')| \lesssim |z_n(t')|^2|v_n(t')|^{p-2}+|z_n(t')|^p$. The Hölder inequality and the Young inequality estimate the term $J_2$ by:
\begin{equation}
    \label{1J2}
    J_2 \leqslant \frac{1}{2} E_n(t) + C\|z_n\|_{L^{p+1}_{T,x}}^{p+1},
\end{equation}

In a similar fashion, one observes that 
\[\int_0^t \int_{\mathbb{R}^3} |\partial _t v_n(t')||z_n(t')|^p \, \mathrm{d}x \, \mathrm{d}t' \lesssim \int_0^t E_n(t')\, \mathrm{d}t' + \|z_n\|_{L^{2p}_{T,x}}^{2p}\] 
and 
\[\int_0^t \int_{\mathbb{R}^3} |\partial _t v_n(t')||v_n|^{p-2}|z_n(t')|^2 \, \mathrm{d}x \, \mathrm{d}t' \lesssim \int_0^t \|z_n(t')\|_{L^{\frac{4(p+1)}{5-p}}}^{2}E_n(t')^{\frac{1}{2}+\frac{p-2}{p+1}}\, \mathrm{d}t'\,.\] Along with the fact that $p<5$ these inequalities yield

\begin{equation}
    \label{1J3}
    J_3 \lesssim \|z_n\|_{L^{2p}_{T,x}}^{2p} + \|z_n\|_{L^{\infty}_TL^{\frac{4(p+1)}{5-p}}_x}^2 \left(1 + \int_0^t E_n(t')\, \mathrm{d}t'\right).
\end{equation}

For the term $J_1$ we write $J_1=\displaystyle\int_0^tK(t')\,\mathrm{d}t'$ where 
\[K(t):=\int_{\mathbb{R}^3}\langle \nabla \rangle \tilde{z}_n(t) |v_n|^{p-1}(t)v_n(t)\,\mathrm{d}x\,. \]
We use the duality between $W^{s_p-1,q}$ and $W^{1-s_p,q'}$ where $\frac{1}{q}+\frac{1}{q'}=1$. This gives
\[K(t) \lesssim \|\langle \nabla \rangle \tilde{z}_n(t)\|_{W^{s_p-1,q}}\||v_n|^{p-1}(t)v_n(t)\|_{W^{1-s_p,q'}}\,,\]
where the constant does not depend on $q$. Using the chain rule in Sobolev spaces, Theorem~\ref{1sobolevChain} yields: 
\[K(t) \lesssim \|\tilde{z}_n(t)\|_{W^{s_p,q}}E_n(t)^{\frac{p-1}{p+1}}\|v_n(t)\|_{W^{1-s_p,\tilde{q}}}\]
where $\frac{1}{q}+\frac{1}{\tilde{q}}+\frac{p-1}{p+1}=1$. Since $p<5$, the term $\|v_n(t)\|_{W^{1-s_p,\tilde{q}}}$ can be interpolated using the Gagliardo-Nirenberg theorem, see Theorem~\ref{1gagliardoNirenberg} applied between $H^1$ and $L^{p+1}$ (a computation ensures that $\alpha < s_p$ in the theorem), and gives
\[K(t) \lesssim \|\tilde{z}_n(t)\|_{W^{s_p,q}} E_n(t)^{1+\alpha (q)}\,,\]
where
\[\alpha (q)=\frac{3(p-1)(2q'-p-1)}{\tilde{q}(p+1)(5-p)}=\frac{3(p-1)}{q(5-p)}=\frac{\beta_p}{q}\,,\]
with $\beta _p := \frac{3(p-1)}{5-p}$ so that 
\begin{equation}
    \label{1J1}
    J_1 \lesssim \int_0^t \|z_n(t')\|_{W^{s_p,q}} E_n(t')^{1+\frac{\beta _p}{q}}\, \mathrm{d}t'
\end{equation}
Finally assembling \eqref{1J1}, \eqref{1J2} and \eqref{1J3} together yields the existence of a universal constant $C=C(T)$ such that 
\begin{align}
    E_n(t) & \leqslant C\left( \|z_n\|^{2p}_{L^{2p}_{T,x}} + \|z_n\|_{L_{T,x}^{p+1}}^{p+1} + \|z_n\|^2_{L^{\infty}_TL_x^{\frac{4(p+1)}{5-p}}} \right) \\ 
    &+ C \|z_n\|^2_{L^{\infty}_TL_x^{\frac{4(p+1)}{5-p}}} \int_0^t E_n(t') \, \mathrm{d}t' +C \int_0^t \notag \|\tilde{z}_n(t')\|_{W^{s_p,q}}E_n(t')^{1+\frac{\beta_p}{q}} \, \mathrm{d}t'\\
    & =: a(z_n)+b(z_n)\int_0^t E_n(t') \, \mathrm{d}t' + \int_0^t c_q(\tilde{z}_n)(t')E_n(t')^{1+\frac{\beta_p}{q}} \, \mathrm{d}t' \notag
\end{align}

We now provide the Yudovich argument, following closely \cite{burqTzvetkov}, which we recalled in the introduction. Set $\lambda_0$ large enough such that the set \[\Omega_0:= \{a(z_n)+b(z_n) \leqslant \lambda_0, \text{ for all } n \geqslant 1\}\] is such that $\mathbb{P}(\Omega _0) \geqslant 1- \frac{\eta}{2}$. Note that such a $\lambda _0$ exists thanks to the conclusion of Proposition~\ref{1propStrichartz} and the fact that $\|\mathbf{P}_n(u_0,u_1)\|_{L^2} \leqslant \|(u_0,u_1)\|_{L^2}$; and $\lambda_0$ depends on $\eta, T$. 

Let $q_0 \geqslant 1$ an integer that will be chosen later. As we have seen in the proof of Proposition~\ref{1propStrichartz}, for all $\lambda >0$ and $p \geqslant q \geqslant q_0$: 
\[\mathbb{P}\left(\|\tilde{z}_n\|_{L^2_TW^{s_p,q}}>\lambda \right) \leqslant \left(\frac{C\sqrt{p}}{\lambda}\right)^p\] where $C=C(\|(u_0,u_1)\|_{\mathcal{H}^{s_p}},T)$ can be chosen independently of $n$ as explained above. For $q \geqslant q_0$ set \[\Omega_q:=\{\|\tilde{z}_n\|_{L^2_TW^{s_p,q}} \leqslant 2 C\sqrt{q}, \text{ for all } n\geqslant 1\}\cap \Omega_0\,.\] When applied with $p=q$ and $\lambda = 2C\sqrt{q}$ we observe that $\mathbb{P}(\Omega_q^{c}) \leqslant 2^{-q}$.

We then use a bootstrap argument. Define \[A_q:=\left\{ t \geqslant 0, \; \sup_{n \geqslant 0}E_n(t) \leqslant \lambda_0^{\frac{q}{\beta _p}}\right\}\] and $t_q^{*}:= \sup A_q$. 
We claim that there is a positive constant $\alpha >0$ which does not depend on $q$ such that 
\begin{equation} \label{1conditionTemps}t_q^{*} \geqslant \alpha q \,.
\end{equation} 
Indeed, for $t \in A_q$ and $n \geqslant 0$ we can write that \[E_n(t) \leqslant \lambda _0+\int_0^t \left(\lambda_0+\|\tilde{z}_n(t')\|_{W^{s_p,q}}\right)E_n(t')\, \mathrm{d}t'\,.\]
Then the Grönwall lemma provides us with
\[E_n(t) \leqslant \lambda_0 \exp \left(\int_0^{t} (\lambda_0+\|\tilde{z}_n(t')\|_{W^{s_p,q}}) \, \mathrm{d}t'\right) \leqslant \lambda _0 \exp \left( \lambda_0 t + 4C\lambda_0\sqrt{qt}\right).\] 
In order to prove \eqref{1conditionTemps} it suffices to exhibit an $\alpha >0$ independent of $q$ such that $t:=\alpha q \in A_q$. A sufficient condition for such an $\alpha$ to exist is to satisfy for every $q \geqslant q_0$ :
\[\lambda_0 \exp(\lambda_0q(\alpha + 4C\sqrt{\alpha})) \leqslant \lambda_0^{\frac{q}{\beta_p}}\,.\] Assume that $\lambda_0 >1$ and $q_0$ large enough to satisfy $\frac{q_0}{\beta_p}-1 >0$. Then it is sufficient for $\alpha$ to satisfy \[\alpha + 4C \sqrt{\alpha} \leqslant \left(\frac{q_0}{\beta_p}-1\right) \frac{\log \lambda_0}{\lambda_0}\] which is satisfied for small $\alpha$ only depending on $C,\lambda_0,q_0,p$. 

Now set $q_0$ larger if needed to ensure $\sum_{q \geqslant q_0} 2^{-q}\leqslant \frac{\eta}{2}$ and thus set $\tilde{\Omega} := \bigcap_{q \geqslant q_0} \Omega_q$, constructed in order to satisfy $\mathbb{P}(\tilde{\Omega}) \geqslant 1-\eta$. Pick $\omega \in \tilde{\Omega}$ and $t \in (0,T)$. 
If $t \leqslant \alpha q_0$ one has $\sup_{n \geqslant 1}  \sup_{t \in (0,T)} E_n(t) \leqslant \lambda_0^{\frac{q_0}{\beta_p}}$, otherwise select a dyadic integer $2^N$ such that $2^N \alpha \leqslant t \leqslant 2^{N+1}\alpha$. Then as $\omega \in \Omega_{2^{N+1}}$ we observe that 
\[\sup_{n \geqslant 1} \sup_{t \in (0,T)} E_n(t) \leqslant \lambda _0^{\frac{2^{N+1}\alpha}{\beta _p}}\,.\] 
Since $t\leqslant T$ and $N \simeq \log t$ we can write it as 
\[\sup_n \sup_{t \in (0,T)} E_n(t) \leqslant C(T,\eta,\|(u_0,u_1)\|_{\mathcal{H}^s})\,,\] 
just as needed.
\end{proof}

\begin{proof}[Sketch of the proof of Corollary~\ref{1coroth2}] Let us explain how one can define an invariant set of full $\mu$-measure, where $\mu \in \mathcal{M}^s$ for $s=\frac{p-3}{p-1}$. We introduce \[\Theta _k ^q := \{(u_0,u_1) \in \mathcal{H}^s, \|\tilde{S}(t)(u_0,u_1)\|_{L^2_TW^{s_p,q}} \leqslant 2Ck\sqrt{q}\} \cap \{ a(S(t)(u_0,u_1)), b(S(t)(u_0,u_1)) < \infty\}\] 
and \[\Sigma := \bigcap_{k \geqslant 1} \bigcap _{q \geqslant q_0} \Theta _k^q \,.\] Then $\Sigma + \mathcal{H}^1$ is of full measure and invariant by the flow. 

For details see the end Section~5 in~\cite{burqTzvetkov}.  
\end{proof}

\begin{remark}\label{1remImpossible} The threshold $s_p=\frac{p-3}{p-1}$ appears to be out of reach when $p\geqslant 5$. Let us explain, at least informally, why a uniform bound on the energy $E$ such as in Proposition~\ref{1propProbabilistic} is out of reach when using the techniques presented here. We start with~\eqref{1eqEnergy} obtained in the introduction:
\[E(t) \simeq  \int_0^t \int_{\mathbb{R}^3} \nabla ^{1-s_p}v(t') v(t')^{p-1} \nabla^{s_p} z(t') \, \mathrm{d}x \, \mathrm{d}t'\,.\]
In order to conclude we would like a bound of the form \[E(t) \lesssim c(q)\int_0^tE(t')^{1+\frac{\beta}{q}}(t')\,\mathrm{d}t'\,,\]
for all large enough $q>1$. Since $(u_0,u_1) \in \mathcal{H}^{s_p}$ and because we did not assume more regularity, and recalling that any $L^{\infty}_T$ or $L^{\infty}_x$ estimates costs derivatives (see Section~\ref{1estimeesProba}), we should use the Hölder inequality putting $\nabla^{s_p}$ in some $L^q$ space instead of $L^{\infty}$. Therefore we have to estimate $\|v^{p-1}\nabla^{1-s_p}v\|_{L^{q'}}$ instead of its $L^1$ norm. We use the Hölder inequality to write:
\[\|v^{p-1}\nabla^{1-s_p}v\|_{L^{q'}} \leqslant \|v^{p-1}\|_{L^{\gamma_1}}\|\nabla ^{1-s_p}v\|_{L^{\gamma_2}}\,,\]
where $\frac{1}{q'}=\frac{1}{\gamma_1}+\frac{1}{\gamma_2}$. 

At this stage, observe that when we estimated the $L^1$ norm (thus $q'=1$) we applied Hölder's inequality with $\gamma_1= \frac{p+1}{p-1}$ and $\gamma_2= \frac{p+1}{2}$. Then we observed the following (schematically):
\begin{enumerate}[label=(\textit{\roman*})]
    \item That $v^{p-1} \in L^{\frac{p+1}{p-1}}$ gives $\|v\|_{L^{p+1}}^{p-1}$, controllable by a power of the energy. 
    \item That $\nabla^{1-s_p} v \in L^{\frac{p+1}{2}}$ may be controlled by the Gagliardo-Nirenberg theorem (which leads to control by a power of the energy), which writes
    \[\|\nabla^{1-s_p} v\|_{L^{\frac{p+1}{2}}} \lesssim \|\nabla v\|_{L^2}^{1-\alpha}\|v\|^{\alpha}_{L^{p+1}}\,,\]
    provided $\alpha \in [0,s_p]$ satisfies 
    \[\frac{2}{p+1}-\frac{1-s_p}{3} = \frac{1-\alpha}{6}+\frac{\alpha}{p+1}\,.\]
    It turned out that $\alpha=s_p$. 
\end{enumerate}

Now observe that when $1$ is increased to $q'$, either $\gamma_1$ or $\gamma_2$ has to increase, in order to match the Hölder indices requirement $\frac{1}{q'}=\frac{1}{\gamma_1}+\frac{1}{\gamma_2}$. Unfortunately, none of these modification lead to a control by the energy. More precisely:
\begin{enumerate}[label=(\textit{\roman*})]
    \item Increasing $\gamma_1$ amounts to estimating $\|v\|_{L^{p+1+\varepsilon}}$ for some $\varepsilon >0$. Since $p\geqslant 5$, we have $p+1+\varepsilon >6$ and then the $L^{p+1+\varepsilon}$ norm is supercritical with respect to the energy, therefore not controllable by a power of the energy. 
    \item Increasing $\gamma_2$ amounts to lowering $\frac{1-\alpha}{6}+\frac{\alpha}{p+1}$, but since $p+1\geqslant 6$ this boils down to increasing $\alpha$, which is not possible since $\alpha$ was already maximal. 
\end{enumerate}
These remarks explain why the Yudovich-Wolibner argument alone fails to obtain the endpoint $s=\frac{p-3}{p-1}$ when $p\geqslant 5$. 
\end{remark}

\appendix

\section{Littlewood-Paley theory and Besov spaces} \label{1appendixA}
This appendix gathers some results dealing with harmonic analysis, analysis in Besov spaces and product laws. A comprehensive treatment of that matter, is given in \cite{bahouriCheminDanchin} and \cite{taylor}. 

We start with a Bernstein-type lemma. For a proof see Lemma 2.1 in \cite{bahouriCheminDanchin}.   

\begin{theorem}[Bernstein-type lemma]\label{1bernstein} Let $(p,q) \in (1,\infty)^2$ with $p \leqslant q$ and $u \in L^p( \mathbb{R}^d)$. Let $B$ be a ball centered on $0$ and $C$ be an annulus, $k \geqslant 0$ an integer. 
\begin{enumerate}[label=(\textit{\roman*})]
    \item If $\hat u$ is supported in $\lambda B$ then $\|\nabla^k u\|_{L^q} \lesssim_k \lambda^{k+d\left(\frac{1}{p}-\frac{1}{q}\right)}\|u\|_{L^p}$.
    \item If $\hat u$ is supported in $\lambda C$ then $\|\nabla^k u\|_{L^p} \simeq_k \lambda^k \|u\|_{L^p}$
    \item The statements (\textit{i}) and (\textit{ii}) are true for non-integer orders of derivation.
\end{enumerate}
\end{theorem}

The proof of Proposition~\ref{1propStrichartzBesov} requires the following consequence of Theorem~\ref{1bernstein}. 

\begin{corollary}\label{1coroBernstein} Let $q \geqslant p > 1$. Let $\psi : \mathbb{R}^d \to \mathbb{R}$ be a smooth function, supported in $[-1,1]^d$. Let $n \in \mathbb{Z}^d$ and $\psi(D-n)$ denote the Fourier multiplier of symbol $\psi(\xi -n)$. Then there holds
\[\|\psi(D-n)u\|_{L^q} \lesssim \|\psi(D -n)u\|_{L^p}\,,\]
for any $u \in L^p(\mathbb{R}^d)$. The implicit constant does not depend on $n$ nor $u$.  
\end{corollary}

\begin{proof} Let $u \in L^p(\mathbb{R}^d)$ and $q \geqslant p >1$. We remark that the support of $\widehat{\psi (D-n)u}$ is contained in $n+[-1,1]^d$. We remark that therefore the support of the Fourier transform of $e^{-in\cdot x} \psi(D-n)u$ is contained in $[-1,1]$ and we can use the Bernstein estimate, Theorem~\ref{1bernstein}~(\textit{i}) with $k=0$ and $\lambda=1$, which writes
\[\|e^{in\cdot x} \psi(D-n)u\|_{L^q} \lesssim \|e^{in\cdot x}\psi(D-n)u\|_{L^p}\,,\]
with implicit constant not depending on $n$ nor $u$. Since the multiplication by $e^{in\cdot x}$ does not affect the $L^p$ norms, the lemma is proved. 
\end{proof}

We now recall the definition of Besov spaces:

\begin{definition} A function $f$ belongs to the nonhomogeneous Besov space $B^{s}_{p,r}$ if, and only if \[\|f\|_{B^s_{p,r}} := \left( \sum_{j \geqslant 0} 2^{jsr} \|\Delta_j f\|_{L^p}^r \right)^{1/r} < \infty.\] Similarly a function $f$ belongs to the homogeneous Besov space $\dot{B}^s_{p,r}$ if, and only if \[\|f\|_{\dot{B}^s_{p,r}} := \left( \sum_{j \in \mathbb{Z}} 2^{jsr} \|\dot{\Delta}_j f\|_{L^p}^r \right)^{1/r} < \infty.\]
\end{definition}

In this text we use the two following reconstruction lemmata, that illustrate the fact that in order to study the $H^s$ norm of a function $f$ it is sufficient to study the frequency localizations $\Delta _j f$. 

\begin{lemma}[\cite{bahouriCheminDanchin}] For $s \in \mathbb{R}$ and $f \in H^s(\mathbb{R}^d)$ one has $\|f\|_{B^s_{2,2}} \sim \|f\|_{H^s}$.
\end{lemma} 

\begin{lemma}[Besov reconstruction, \cite{bahouriCheminDanchin}]\label{1lemmaReconstruction} Let $s>0$ and $p,r \in [1,\infty]$. Let $(u_j)_{j \geqslant 0}$ be a sequence of smooth functions which satisfy \[\left( \sup_{|\alpha| \leqslant \lfloor s\rfloor +1} 2^{j(s-|\alpha|)} \|\partial ^{\alpha} u_j\|_{L^p}\right)_{j \geqslant 0} \in \ell ^r,\] then one has $\displaystyle u = \sum_{j \geqslant 0}  u_j \in B^s_{p,r}$ and the estimate: \[\|u\|_{B^s_{p,r}} \lesssim \left \| \left( \sup_{|\alpha| \leqslant \lfloor s\rfloor +1} 2^{j(s-|\alpha|)} \|\partial ^{\alpha} u_j\|_{L^p}\right)_{j \geqslant 0} \right \|_{\ell ^r},\] whith implicit constant only depending on $s$. The same result holds for homogeneous Besov spaces $\dot B^{s}_{p,r}$ and the $\ell ^r$ norm taken with indices running in $\mathbb{Z}$.
\end{lemma}

Next we recall some embedding theorems: 

\begin{theorem}[Sobolev-Besov embeddings, \cite{bahouriCheminDanchin}] \label{1sobolev} Let $1 \leqslant p_1 \leqslant p_2 \leqslant \infty$ and $1 \leqslant r_1 \leqslant r_2 \leqslant \infty$, $p,q \in [1,\infty]$ and $s,s_1,s_2 \in \mathbb{R}$. Then:
\begin{enumerate}[label=(\textit{\roman*})]
    \item The embedding $\displaystyle B^{s}_{p_1,r_1} \hookrightarrow B^{s-d\left(\frac{1}{p_1}-\frac{1}{p_2}\right)}_{p_2,r_2}$ is continuous.
    \item For $1 \leqslant p \leqslant \infty$ the embeddings $B^s_{p,1} \hookrightarrow W^{s,p} \hookrightarrow B^s_{p,\infty}$ are continuous. The homogeneous counterpart is also true.  
    \item The embedding $W^{s_1,p} \hookrightarrow W^{s_2,q}$ is continuous as soon as $\frac{1}{p}-\frac{s_1}{d} \leqslant\frac{1}{q} - \frac{s_2}{d}$.
    \item The embedding (\textit{iii}) is locally compact whenever the inequality is strict.
    \item Let $r_1>r_2$, $s<0$, $p \in [1,\infty]$. Then for arbitrarly small $\varepsilon >0$, the embedding $B_{p,r_1}^{s+\varepsilon} \hookrightarrow B_{p,r_2}^s$ is continuous. 
\end{enumerate}
\end{theorem}

The next theorem is a well-known interpolation inequality.

\begin{theorem}[Gagliardo-Nirenberg]\label{1gagliardoNirenberg} Let $p_0,p_1  \in (1,\infty)$ and $s,t >0$. Let $p>1$ and $\alpha \in (0,1)$ satisfying 
\[\displaystyle -\frac{s}{d}+\frac{1}{p}=(1-\alpha)\left(\frac{1}{p_0} -\frac{t}{d}\right) + \frac{\alpha}{p_1} \text{ and } s\leqslant (1-\alpha)t\,.\]
Then for $u \in W^{t,p_0}(\mathbb{R}^d) \cap L^{p_1}(\mathbb{R}^d)$ one has \[ \| u\|_{\dot{W}^{s,p}} \lesssim \|u\|^{1-\alpha}_{\dot{W}^{t,p_0}} \|u\|_{L^{p_1}}^{\alpha}\,,\]
and the same result holds in nonhomogeneous Sobolev spaces.
\end{theorem}

\begin{proof} The limit embedding is given by Theorem 2.44 in \cite{bahouriCheminDanchin}, that is: 
\[\|u\|_{\dot{W}^{s,\bar{p}}} \lesssim \|u\|_{\dot{W}^{t,p_0}}^{s/t} \|u\|^{1-s/t}_{L^{p_1}}\] where $\frac{1}{\bar{p}} = \frac{s/t}{p_0}+\frac{1-s/t}{p_1}$. On the other hand the Sobolev embedding reads $\|u\|_{\dot{W}^{s,\tilde{p}}} \lesssim \|u\|_{\dot{W}^{t,p_0}}$ where $\frac{1}{\tilde{p}}-\frac{s}{d}=\frac{1}{p_0}-\frac{t}{d}$. Let $p$ a real number such that there exists $\tilde{\theta} \in [0,1]$ satisfying $\frac{1}{p}=\frac{\tilde{\theta}}{\bar{p}}+\frac{1-\tilde{\theta}}{\tilde{p}}$. Then by interpolation and using the previous inequalities on have 
\[\|u\|_{\dot{W}^{s,p}} \leqslant \|u\|_{\dot{W}^{s,\bar{p}}}^{\tilde{\theta}}\|u\|_{\dot{W}^{s,\tilde{p}}}^{1-\tilde{\theta}} \lesssim \|u\|^{1-\theta}_{\dot{W}^{t,p_0}} \|u\|_{L^{p_1}}^{\theta}\] with $\theta = \left(1-\frac{s}{t}\right) \tilde{\theta}$. A straightforward computation yields 
\begin{equation}
\label{1eqExposants}
\frac{1}{p} - \frac{s}{d}=(1-\theta)\left(\frac{1}{p_0}-\frac{t}{d} \right) + \frac{\theta}{p_1}. 
\end{equation}
Note that the imposed condition $\theta \in \left[0,1-\frac{s}{t}\right]$ comes from the fact that $\tilde{\theta} \in [0,1]$. All the range of $p$ is covered since the equality \eqref{1eqExposants} implies $\frac{1}{p} \in \left[\frac{1}{\bar{p}},\frac{1}{\tilde{p}}\right]$.

The proof in the nonhomogeneous case follows using $\|u\|_{W^{s,q}} \leqslant \|u\|_{L^q} + \|u\|_{\dot{W}^{s,q}}$. 
\end{proof}

\begin{theorem}[Sobolev chain rule, \cite{taylor}]\label{1sobolevChain} For $s \in (0,1)$, $p \in (1,\infty)$ and a function $F\in \mathcal{C}^1(\mathbb{R})$ such that $F(0)=0$ and such that there is a $\mu \in L^1([0,1])$ such that for every $\theta \in [0,1]$: \[|F'(\theta v +(1-\theta)w)| \leqslant \mu(\theta) \left( G(v)+G(w) \right)\] where $G>0$. We have $\|F \circ u \|_{W^{s,p}} \lesssim \|u\|_{W^{s,p_0}}\|G \circ u\|_{L^{p_1}}$ as soon as $\frac{1}{p_0}+\frac{1}{p_1}=\frac{1}{p}$, provided $p_0 \in (1,\infty]$ and $p_1 \in (1,\infty)$.  
\end{theorem}

In order to derive product laws in Sobolev or Besov spaces, recall the paraproduct algorithm introduced by J.M. Bony, which consists in writing for two given functions $a,b$: \[ ab=T_ab+T_ba+R(a,b)=\dot{T}_ab+\dot{T}_ba+\dot{R}(a,b)\] where \[ T_ab:=\sum_{j}\mathbf{P}_{j-1}a\Delta_jb \text{ and } R(a,b):=\sum_{|k-j| \leqslant 1} \Delta _k a \Delta_j b \,,\] and \[\dot{T}_ab:=\sum_{j}\dot{\mathbf{P}}_{j-1}a\dot{\Delta}_jb \text{ and } \dot{R}(a,b):=\sum_{|k-j| \leqslant 1} \dot{\Delta} _k a \dot{\Delta}_j b \,,\] respectively denote nonhomogeneous paraproduct and remainder (resp. homogeneous). 

\begin{theorem}[Tame estimates, \cite{bahouriCheminDanchin}]\label{1productBesov} Let $s>0$, $p,r \in [1,\infty]$, and $p_1,p_2 \in [1, \infty]$ satisfying $\frac{1}{p_1}+\frac{1}{p_2}=\frac{1}{p}$. Then: 
\begin{enumerate}[label=(\textit{\roman*})]
    \item $\|\dot{T}_u v\|_{\dot B^s_{p,r}} \lesssim \|u\|_{L^{p_1}}\|v\|_{\dot B^{s}_{p_2,r}}$.
    \item $\|\dot R(u,v)\|_{\dot B^s_{p,r}} \lesssim \|u\|_{L^{p_1}}\|v\|_{\dot B^{s}_{p_2,r}}$.
\end{enumerate}
\end{theorem}

From this one infers the next inequality which is important in our analysis:

\begin{corollary}\label{1coroProductRule} Let $s>0$, $p,r \in [1,\infty]$, and $p_1,p_2,p'_1,p'_2 \in [1, \infty]$ satisfying \[\frac{1}{p_1}+\frac{1}{p_1'}=\frac{1}{p_2}+\frac{1}{p_2'}=\frac{1}{p}\,.\] 
Then: 
\[\|fg\|_{\dot B^{s}_{p,r}} \lesssim \|f\|_{\dot B^{s}_{p_1,r}} \|g\|_{L^{p_1'}} + \|f\|_{L^{p_2}} \|g\|_{\dot B^{s}_{p_2',r}}\,.\] 
\end{corollary}

\bibliographystyle{alpha}
\bibliography{biblio}

\begin{thebibliography}{BOP15b}

\bibitem[BB14a]{bourgainBulut1}
Jean Bourgain and Aynur Bulut.
\newblock Almost sure global well posedness for the radial nonlinear
  schrödinger equation on the unit ball i: the 2d case.
\newblock {\em Ann. Inst. H. Poincaré Anal. Non Linéaire}, 31:1267–1288,
  2014.

\bibitem[BB14b]{bourgainBulut2}
Jean Bourgain and Aynur Bulut.
\newblock Almost sure global well posedness for the radial nonlinear
  schrödinger equation on the unit ball ii: the 3d case.
\newblock {\em J. Eur. Math. Soc. (JEMS)}, 16:1289–1325, 2014.

\bibitem[BCD11]{bahouriCheminDanchin}
H.~Bahouri, J.Y. Chemin, and R.~Danchin.
\newblock {\em Fourier Analysis and Nonlinear Partial Differential Equations}.
\newblock Grundlehren der mathematischen Wissenschaften. Springer Berlin
  Heidelberg, 2011.

\bibitem[BF13]{boyer}
Franck Boyer and Pierre Fabrie.
\newblock {\em Mathematical Tools for the Study of the Incompressible
  Navier-Stokes Equations and Related Models}.
\newblock Applied Mathematical Science. Springer-Verlag, 2013.

\bibitem[BOP15a]{benyiOhPocovnicu}
\'Arp\'ad Bényi, Tadahiro Oh, and Oana Pocovnicu.
\newblock Wiener randomization on unbounded domains and an application to
  almost sure local well-posedness of nls.
\newblock {\em Excursions in Harmonic Analysis}, 4:3--25, 2015.

\bibitem[BOP15b]{benyiOhPocovnicu2}
\'Arpad Bényi, Tadiharo Oh, and Oana Pocovnicu.
\newblock On the probabilistic cauchy theory of the cubic nonlinear
  schrödinger equation on $\mathbb{R}^d$, $d \ge 3$.
\newblock {\em Trans. Amer. Math. Soc. Ser. B}, 2:1--50, 2015.

\bibitem[Bou96]{bourgain}
J.~Bourgain.
\newblock Invariant measures for the 2d-defocusing nonlinear schrödinger
  equation.
\newblock {\em Comm. Math. Phys.}, 176:421--445, 1996.

\bibitem[BT08a]{burqTzvetkov2}
Nicolas Burq and Nikolay Tzvetkov.
\newblock Random data cauchy theory for supercritical wave equations i: Local
  theory.
\newblock {\em Inventiones Mathematicae}, 173:449--475, 2008.

\bibitem[BT08b]{burqTzvetkov3}
Nicolas Burq and Nikolay Tzvetkov.
\newblock Random data cauchy theory for supercritical wave equations ii: A
  global result.
\newblock {\em Inventiones Mathematicae}, 173:477--496, 2008.

\bibitem[BT14]{burqTzvetkov}
Nicolas Burq and Nikolay Tzvetkov.
\newblock Probabilistic well-posedness for the cubic wave equation.
\newblock {\em J. Eur. Math. Soc.}, 16(1):1--30, 2014.

\bibitem[BTT15]{burq}
Nicolas Burq, Laurent Thomann, and Nikolay Tzvetkov.
\newblock Global infinite energy solutions for the cubic wave equation.
\newblock {\em Bulletin de la Société Mathématique de France},
  143(2):301--313, 2015.

\bibitem[Hö65]{hormander}
L.~Hörmander.
\newblock Pseudo-differential operators and hypoelliptic equations.
\newblock {\em Proc. Symp. Pure math}, X:138--183, 1965.

\bibitem[LM14]{lurmanMendelson}
Jonas Lührmann and Dana Mendelson.
\newblock Random data cauchy theory for nonlinear wave equations of power-type
  on $\mathbb{R}^3$.
\newblock {\em Comm. Partial Differential Equations}, 39:2262--2283, 2014.

\bibitem[LM16]{lurmannMenselsonBis}
Jonas Lührmann and Dana Mendelson.
\newblock On the almost sure global well-posedness of energy sub-critical
  nonlinear wave equations on $\mathbb{R}^3$.
\newblock {\em New York J. Math.}, 22:209--227, 2016.

\bibitem[NPS13]{nahmodPavlovicStaffilani}
A.R. Nahmod, N.~Pavlovi\'c, and G.~Staffilani.
\newblock Almost sure existence of global weak solutions for supercritical
  navier-stokes equations.
\newblock {\em SIAM J. Math. Anal.}, 45:3431--3542, 2013.

\bibitem[OOP17]{okamotoOhPocovnicu}
M.~Okamoto, Tadiharo Oh, and Oana Pocovnicu.
\newblock On the probabilistic well-posedness of the nonlinear schrödinger
  equation with non-algebraic nonlinearities.
\newblock {\em arXiv:1708.01568}, 2017.

\bibitem[OP16]{ohPocovnicu}
Tadahiro Oh and Oana Pocovnicu.
\newblock Probabilistic global well-posedness of the energy-critical defocusing
  quintic nonlinear wave equation on $\mathbb{R}^3$.
\newblock {\em Journal de Mathématiques Pures et Appliquées}, 105:342--366,
  2016.

\bibitem[OP17]{ohPocovnicu2}
Tadahiro Oh and Oana Pocovnicu.
\newblock A remark on almost sure global well-posedness of the energy-critical
  defocusing nonlinear wave equations in the periodic setting.
\newblock {\em Tohoku Math. J.}, 69(3):455--481, 2017.

\bibitem[Poc17]{pocovnicu}
Oana Pocovnicu.
\newblock Almost sure global well-posedness for the energy-critical defocusing
  nonlinear wave equation on $\mathbb{R}^d$, $d=4,5$.
\newblock {\em J. Eur. Math. Soc.}, 19:2321--2375, 2017.

\bibitem[Str70]{strauss}
Walter~A. Strauss.
\newblock On weak solutions of semi-linear hyperbolic equations.
\newblock {\em Anais. Acad. Brasil. Cienc.}, 42:645--651, 1970.

\bibitem[SX16]{sunXia}
Chenmin Sun and Bo~Xia.
\newblock Probabilistic well-posedness for supercritical wave equations with
  periodic boundary condition on dimension three.
\newblock {\em Illinois J. Math.}, 60(2):481--503, 2016.

\bibitem[Tay07]{taylor}
M.E. Taylor.
\newblock {\em Tools for PDE: Pseudodifferential Operators, Paradifferential
  Operators, and Layer Potentials}.
\newblock Mathematical surveys and monographs. American Mathematical Society,
  2007.

\bibitem[Wie32]{wiener}
Norbert Wiener.
\newblock Tauberian theorems.
\newblock {\em Annals of Mathematics}, 33:1--100, 1932.

\bibitem[Wol33]{wolibner}
W.~Wolibner.
\newblock Un théorème sur l’existence du mouvement plan d’un fluide
  parfait, homogène, incompressible, pendant un temps infiniment long.
\newblock {\em Math. Z.}, 37:698--726, 1933.

\bibitem[Yud63]{yudovich}
V.I. Yudovich.
\newblock Non-stationary flows of an ideal incompressible fluid.
\newblock {\em \u{Z} Vy\u{c}isl. Mat. i Mat. Fiz.}, 3:1032--1066, 1963.

\end{thebibliography}
\end{document}